\documentclass[11pt]{amsart}

\newtheorem{theorem}{Theorem}[section]
\newtheorem{lemma}[theorem]{Lemma}
\newtheorem{corollary}[theorem]{Corollary}

\theoremstyle{definition}

\newtheorem{proposition}[theorem]{Proposition}
\theoremstyle{remark}
\newtheorem{remark}[theorem]{Remark}

\numberwithin{equation}{section}

\def\epi{\varepsilon}
\def\R{\mathbb{R}}

\newcommand{\p}{\partial}
\newcommand{\wt}{\widetilde}

\newcommand{\CT}{\mathcal{T}}

\newcommand{\CX}{\mathcal{X}}
\newcommand{\CM}{\mathcal{M}}
\newcommand{\CH}{\mathcal{H}}
\newcommand{\CW}{\mathcal{W}}
\newcommand{\CG}{\mathcal{G}}
\newcommand{\CF}{\mathcal{F}}

\newcommand{\CL}{\mathcal{L}}

\newcommand{\Pe}{\Pi^{e}_{c,y_{1}}}
\newcommand{\Pee}{\Pi^{e}_{c,y_{2}}}
\newcommand{\J}{JL_{c,y_{1}}}
\newcommand{\JJ}{JL_{c,y_{2}}}
\newcommand{\BR}{\mathbb{R}}

\usepackage{amsmath,amsthm,verbatim,amssymb,amsfonts,amscd, graphicx,txfonts,paralist}
\usepackage{graphics}
\usepackage{fourier}
\begin{document}

\title[Center Manifold]{Dynamics near the solitary waves of the supercritical \MakeLowercase{g}KDV Equations}

\author{Jiayin Jin} 
\address{School of Mathematics, Georgia Institute of Technology}
\email{jin@math.gatech.edu}
\author{Zhiwu Lin${}^\dagger$}
\address{School of Mathematics, Georgia Institute of Technology}
\email{zlin@math.gatech.edu}
\thanks{${}^\dagger$Work supported in part by NSF-DMS 1411803 and NSF-DMS 1715201}
 \author{Chongchun Zeng${}^\ddagger$}
 \address{School of Mathematics, Georgia Institute of Technology}
\email{chongchun.zeng@math.gatech.edu}
\thanks{${}^\ddagger$Work supported in part by NSF-DMS 1362507}

%    General info
\subjclass[2000]{}

\date{}

\dedicatory{}


\begin{abstract}
This work is devoted to study the dynamics of the supercritical gKDV equations near solitary waves in the energy space $H^1$. We construct smooth local center-stable, center-unstable and center manifolds near  the manifold of solitary waves and give a detailed description of the local dynamics near solitary waves. In particular, the instability is characterized as following: any forward flow not starting from the center-stable manifold will leave a neighborhood of the manifold of solitary waves exponentially fast. Moreover, orbital stability is proved on the center manifold, which implies the uniqueness of the center manifold and the global existence of solutions on it. 
\end{abstract}

\maketitle

%%%%%%%%%%%%%%%%%%%%%%%%%%%%%%%%%%%%%%%%%%%%%%%%%%%%%%%%%%%%%%%%%%%%%%%%
%\footnote{Here is an example of a footnote. Notice that this footnote
%text is running on so that it can stand as an example of how a footnote
%with separate paragraphs should be written.
%\par
%And here is the beginning of the second paragraph.}%
%%%%%%%%%%%%%%%%%%%%%%%%%%%%%%%%%%%%%%%%%%%%%%%%%%%%%%%%%%%%%%%%%%%%%%%%

\section{Introduction}
We consider the supercritical gKDV equation
\begin{equation}\label{gkdv}
u_{t}+(u_{xx}+u^{k})_{x}=0, \qquad k>5, \qquad u\in H^{1}(\R).
\end{equation}
The cases of the integer $k<5$, $k=5$, and $k>5$  are referred to as the subcritical, critical,  and supercritical cases, respectively. The well-posedness of \eqref{gkdv} is classical (see \cite{KPV} and \cite{K}). The subcritical gKDV equation are globally well-posed in $H^{1}$, while the critical and supercritical gKDV are locally well-posed. 
 
Blow-up solutions have been obtained in the critical case by Martel and Merle \cite{MM1} and the slightly supercritical case of $5<k<5+\epsilon$ by Lan \cite{Lan}.

The gKDV equation has a Hamiltonian form $u_{t}=JE'(u),$ where $J=\partial_{x}$ is the symplectic operator and $$E(u)=\int_{\mathbb{R}}\frac{1}{2}u^{2}_{x}-\frac{1}{k+1}u^{k+1}dx$$ is the conserved energy. Due to the translationl invariance, the momentum $$P(u)=\frac{1}{2}\int_{\mathbb{R}}u^{2} dx$$ is also conserved. Moreover, the gKDV equation is  invariant under the scaling 
\begin{equation} \label{E:scaling}
(\CT^{\lambda}u)(t,x)= \displaystyle \lambda^{\frac{2}{k-1}}u(\lambda^{3}t, \lambda x).
\end{equation}

The linear dispersion and nonlinear effect interact to produce  solitary waves,  $u_{c}(x,t)= Q_{c}(x-ct)$, where $$ Q_{c}(x)=\left(\CT^{\sqrt{c}}Q\right)(x)=c^{\frac{1}{k-1}}Q(\sqrt{c}x)$$ 
with 
$$Q(x)=\Big(\frac{k+1}{2}sech^{2}\big(\frac{k-1}{2}x\big)\Big)^{\frac{1}{k-1}} \in H^1$$ 
being the unique positive even solution to 
\begin{equation}
Q_{xx}-Q+Q^{k}=0, \quad Q(\pm \infty) =0. 
\end{equation}

These solitary waves  play a fundamental role in the dynamics of the gKDV equation. The stability of the solitary waves has been studied extensively.   For the subcritical gKDV equation,  solitary waves are orbitally stable, see \cite{B,BSS,Bo,W}. Furthermore, for $k=2,3$ Pego and Weinstein \cite{PW} proved asymptotic stability of the whole family of solitary waves for initial data with exponential spatial decay at $\infty$.  Mizumachi\cite{Mi} proved asymptotic stability of  the whole family of solitons for initial date with algebraic spatial decay at $\infty$ for $k=2,3,4$. Martel and Merle \cite{MM2} proved asymptotic stability in weak topology for the subcritical gKDV equation for initial data in $H^{1}$, that is for any $\delta>0$, there exists $\alpha$, such that for any $u_{0}$ satisfying $\|u_{0}-Q_{c}\|_{H^{1}}\le \alpha$, there exists $c(t), x(t)$, such that $\left(u(t,\cdot+x(t))-Q_{c(t)}\right) \rightharpoonup 0$ in $H^1$ as $t\rightarrow \infty$. 

For the critical case, in a series of works \cite{MMR1, MMR2,MMR3}, Martel, Merle and Rapha{\"e}l classified the dynamics for a set of initial data $$\mathcal{A}=\{u_{0}=Q+v: \|v\|_{H^{1}}\le \alpha_{0}, \int_{x>0}x^{10}v^{2}(x)dx<1\}.$$ More specifically, the solutions with initial data in $\mathcal{A}$ are classified into three classes: (i) blow up in finite time; (ii) exist globally in time and stay close to the orbits of solitary waves for any $t>0$; (iii) exist globally and exit a neighborhood of the traveling wave manifold. 

Recently, Martel, Merle, Nakanishi and Rapha{\"e}l constructed a co-dimension 1 threshold manifold separating the initial data satisfying (i) and (iii), and showed that the solutions with initial data on the threshold manifold belong to (ii).

For the supercritical gKDV equations, Bona, Souganidis and Strauss \cite{BSS} proved the solitary waves are orbitally unstable. Namely there exist solutions starting arbitrarily close to the traveling wave manifold, but eventually go away. Combet \cite{Com} constructed special solutions converge to solitary waves exponentially fast as $t\rightarrow +\infty$ in $H^{1}$. 

Naturally, one may raise the question: whether there exist solutions starting near solitary waves behaving differently than the above two types?  Furthermore, how are all these different type of solutions organized/located in the energy space $H^1$ near the traveling waves?

In this work, we give a detailed description of the local dynamics of the supercritical gKDV equation near the soliton manifold 
$$\mathcal{M}=\{Q_{c}(\cdot+y): c>0, y\in \mathbb{R}\}. $$
\begin{enumerate}
\item {\it Existence (Section 4) and smoothness (Section 5) of local invariant manifolds of $\CM$ in $H^{1}$}: 
\begin{itemize}
\item There exist  co-dimension $1$ center-stable and center-unstable manifolds $\CW^{cs}(\mathcal{M})$ and $\CW^{cu}(\mathcal{M})$ of $\mathcal{M}$, respectively, such that $\CM \subset \CW^{cs, cu}$ and for any $m\ge 1$, there exist neighborhoods of $\CM$ where $\CW^{cs, cu} (\CM)$ are $C^m$ submanifolds. 
\item Moreover, $\CW^{cs}(\CM)$ and $ \CW^{cu}(\CM)$intersect transversally along the center manifold $\CW^c(\CM) =\CW^{cs}(\CM) \cap \CW^{cu}(\CM)$ which is a smooth co-dimension 2 submanifold. 
\item $\CW^{cs, cu, c}(\CM)$ are invariant under spatial translation and rescaling \eqref{E:scaling}.
\item These manifolds $\CW^{cs, cu, c}(\CM)$ are locally invariant under the flow of \eqref{gkdv}. Namely, an orbit starting on $\CW^{cs, cu, c}(\CM)$  can leave them only through their boundaries.

\end{itemize}

\item ({\it Local dynamics near traveling waves}) 
\begin{itemize}
\item $\CW^{cs} (\CM)$ (or $\CW^{cu} (\CM)$, or $\CW^{c} (\CM)$, respectively) is the set of initial data whose orbits under \eqref{gkdv} stay close to $\CM$ for all $t\ge 0$ (or $t\le 0$, or $t\in \R$, respectively). ({\bf Propositions} \ref{characterization}, \ref{offcs}, and \ref{offcu}) 
\item If the initial data is not on $\CW^{cs}(\mathcal{M})$ (or $\CW^{cu}(\mathcal{M})$), then the forward (or backward) orbit  exits a neighborhood of $\mathcal{M}$ exponentially fast. {\bf Propositions} \ref{offcs} and \ref{offcu})

\item $\CW^c(\CM)$ is exponentially stable on $\CW^{cs}(\CM)$ as $t\to +\infty$ and on $\CW^{cu}(\CM)$ as $t\to -\infty$. ({\bf Propositions} \ref{attract-center})
\item $\CM$ is orbitally stable on $\CW^c(\CM)$ in the sense that, for any neighborhood $\mathcal{U}\subset \CW^c(\CM)$ of $\CM$, there exists a neighborhood $\mathcal{V} \subset \mathcal{U}$ such that orbits starting in $\mathcal{V}$ stay in $\mathcal{U}$ for all $t \in \R$. ({\bf Propositions} \ref{attract-center})\end{itemize}
\end{enumerate}

\begin{remark} 
In this paper we focus on the center-stable, center-unstable, and center manifolds of the 2-dimensional traveling wave manifold $\CM$. The stable and unstable manifold of the latter should follow from an easier (see Remark \ref{R:stable-unstable}) construction and would be carried out in a separate paper.
\end{remark} 

Let us briefly outline our proof.  As a convention, we write the gKDV equation in the traveling frame $(t, x-ct)$ with a fixed wave speed $c$ and let $u(t,x)= U(t,x-ct)$, where $U(t,x)$ satisfies $$U_{t}-cU_{x}+(U_{xx}+U^{k})_{x}=0. \hspace{0.1in}(gKDV-Tr)$$ Clearly $Q_c$, the profile of the traveling wave, is an equilibrium of (gKDV-Tr).  The linearization of (gKDV-Tr) at $Q_{c}$ takes the Hamiltonian form of 
\[
U_{t}=JL_{c}U, \quad \text{ where } \quad J=\partial_{x}, \; \;  L_{c}=c-\partial_{xx}-kQ_{c}^{k-1}.  
\]
Thanks to this Hamiltonian structure,  the energy space $H^{1}$ can be decomposed into three invariant subspaces (under the linearized flow $e^{tJL_{c}}$) 
$$X=X^{+}\oplus X^{-}\oplus X^{c}, \quad \text{ where } X^\pm = span\{V^\pm\}, \; \; JL_c V^\pm =\pm \lambda_c V^\pm, \; \; \lambda_c>0.$$ 
Here 
$X^{c}=X^{T}\oplus X^{e}$ is the center space, where $X^{T}=span\{\partial_{x}Q_{c}\} = \ker(JL_c)$ and $L_{c}$ is uniformly positive definite on $X^{e}$. Furthermore, the following trichotomy holds
\begin{equation}
\begin{split}
&\Vert e^{JL_{c}t}\vert_{X^\pm}\Vert \le e^{\pm \lambda_c t}, \hspace{0.1 in } \text{for}\; \mp t\ge 0\\
&\Vert e^{JL_{c}t}\vert_{X^{c}}\Vert \le M(1+t), \hspace{0.1 in } \text{for}\; t\in \mathbb{R}.
\end{split}
\end{equation}

The linearized dynamic structure described by this trichotomy serves as the cornerstone of the study on the  nonlinear dynamics, with the bridge classically provided by the invariant manifold theory for ODEs and PDEs (mainly for semilinear PDEs). Roughly, the linear trichotomy in the phase space along with nonlinear terms being smooth mappings from the phase space to itself imply that there exist nonlinearly locally invariant submanifolds very close to the invariant subspaces. 
However, this classical theory  does not apply to gKDV directly as its nonlinearity  contains a loss of derivative.  

Fortunately, the linear flow $e^{-t\p_{xxx}}$ has a smoothing effect, which may still allow the stable and unstable manifolds of $Q_c$ to be constructed through a modification to the classical approaches. Since the stable manifold and unstable manifold are unique for each $Q_{c}(\cdot+y)$ and extend in transversal directions of $\CM$, therefore one can construct the stable and unstable manifolds for $Q_{c}$ first and then translate them along $\mathcal{M}$ to form the stable and unstable manifolds of the whole $\mathcal{M}$.

Compared to stable and unstable manifolds, there is an additional difficulty in the construction of invariant manifolds containing center directions. 
Unlike stable and unstable manifold, center manifolds usually are not unique and they extend in the direction of $\CM$, therefore one can not translate center manifolds of $Q_{c}$ to obtain the ones of $\mathcal{M}$. As $\mathcal{M}$ should be contained in the center manifolds, so it is reasonable to attempt to construct the center manifolds of the whole $\mathcal{M}$  directly. This brings up an issue how to set up a suitable coordinate system in a neighborhood of $\mathcal{M}$.  A tempting choice is to use the translational parametrization to write any $U$ in a tubular neighborhood of $\mathcal{M}$ as 
$$U=\phi(y,a^+,a^-, V^{e})=(Q_{c}+a^+V^{+}+a^-V^{-}+W^{e})(\cdot+y), \quad W^e \in X^e.$$ 

However, this translational parametrization is not smooth in $H^{1}$. To see this, we take the derivative of $\phi$ with respect to $y$ and a term $\partial_{y}V^{e}(\cdot+y)\in L^2$ appears, while the other terms in $\partial_{y}\phi$ are regular enough. This was also the main difficulty in the work of Nakanishi and Schlag\cite{NS1}, where the authors constructed the center-stable manifolds of the manifold of ground states for the Klein-Gordon equation. They constructed a clever nonlinear  ``optimal mobile distance"  to overcome this difficulty. In this paper, we follow the approach as in \cite{BLZ08, JLZ} to utilize a smooth bundle coordinate system. Namely, any $U$ in a tubular neighborhood of $\mathcal{M}$ is written as 
$$U=\psi(y,a^+, a^-,V^{e})= (Q_{c}+ a^+V^{+}+a^-V^{-})(\cdot+y)+V^{e}, \quad  V^e \in X_y^e=\{u: u(\cdot-y)\in X^{e}\}.$$ 
 
Since $Q_c$ and $V^\pm$ are 
smooth functions, the corresponding projection $\Pi^e_{y}: H^1 \to X_y^e$ with $\ker \Pi_y^e = span\{\p_x Q_c (\cdot +y), V^\pm(\cdot +y)\}$ is smooth in $y$. Consequently, 
$\widetilde{X}^{e}= \{(y,V): V\in X_{y}^{e}\}$ is a smooth bundle over $y\in \mathbb{R}$. We rewrite (gKDV-Tr) using this smooth bundle coordinate system $\psi$. Even though some geometric notions are involved, we still manage to obtain certain desired smoothing estimates.  

 Then we are able to perform Lyapunov-Perron method to construct invariant manifolds of the soliton manifold, which help to reveal a rather complete picture of the local dynamic structure near the soliton manifold. In particular, the orbital stability of $\CM$ on the center manifold is obtained from a Lyapunov functional argument based on the fact that $Q_{c}$ is a critical point of the energy momentum functional $E-cP$ whose Hessian is uniformly positive definite in $X^e$.  
The orbital stability on center manifolds yields characterizations (Proposition \ref{characterization}) of the center-stable, center-unstable, and center manifolds of $\CM$, which in turn lead to their local uniqueness. 

Consequently, any solution $u(t,x)$ on the center manifold close to $\CM$ satisfies the assumption in Theorem 1 in \cite{MM3} and thus there exist $c_0 >0$ and functions $c(t)$ and $\rho (t)$,  $t>0$, such that 
\[
\Vert u(t) - Q_{c(t)} (\cdot - \rho(t)\Vert_{H_x^1 (x >  \frac {c_0}{10} t)} \to 0, \quad \text{ as } t \to \infty.  
\]

There are some previous results on the construction of invariant manifolds for semilinear PDEs. Bates and Jones \cite{BJ} proved a general theorem for the existence of local invariant manifolds of equilibria for semilinear PDEs by the graph transform, and then applied it to the Klein-Gordon equation in the radial setting. In \cite{Sc}, Schlag constructed a co-dimension $1$ center-stable manifold of the manifold of ground states for 3D cubic NLS in $W^{1,1}(\mathbb{R}^{3})\bigcap W^{1,2}(\mathbb{R}^{3})$ under an assumption that the linearization of NLS at each ground state has no eigenvalue embedded in the essential spectrum and proved scattering on the center-stable manifold.  Later, this result was improved by Beceanu \cite{Be1, Be2} who constructed  center-stable manifolds in $W^{1,2}(\mathbb{R}^{3})\bigcap |x|^{-1}L^{2}(\mathbb{R}^{3})$ and in critical space $\dot{H}^{1/2}(\mathbb{R}^{3})$. Similar results were obtained in Krieger and Schlag \cite{KS} for the supercritical 1D NLS. Nakanishi and Schlag \cite{NS4} constructed a center-stable manifold of ground states for 3D cubic NLS in the energy space with a radial assumption by using the framework in Bates and Jones \cite{BJ}. Nakanishi and Schlag \cite{NS1} constructed center-stable manifolds of ground states for nonlinear Klein-Gordon equation without radial assumption. Also, see \cite{NS3, KNS1,KNS2,KNS3} for related results. To the best of our knowledge, this current work is the first one constructing invariant manifolds of a global soliton manifold for a dispersive PDE with derivative nonlinearities. Our approach, involving using the bundle coordinates and deriving space-time estimates with small exponential growth, seems to be rather general and, with minimal essential modifications, applicable to unstable relative equilibria (including ground and excited states) of a class of Hamiltonian PDEs with natural symmetries (see also \cite{JLZ}). 

This paper is organized as following. In Section \ref{bundlecoordinates}, we establish bundle coordinates over the soliton manifold and rewrite the equations. In Section \ref{LinearAnalysis}, we derive smoothing space-time estimates in the bundle coordinates and then prove several apriori estimates. In Section \ref{construction}, we construct Lipschitz invariant manifolds of the soliton manifold, whose smoothness is proved in Section \ref{smoothness of invariant manifolds}. 
In Section \ref{classification}, we analyze the local dynamics near soliton manifold by invariant manifolds.  

\noindent {\bf A remark on notations.} Throughout the paper, $\langle \cdot, \cdot \rangle$ denotes the dual pairing between elements of a Banach space and its dual space. The generic upper bound $C$ may depend on $c>0$, but not other phase space variables or parameters, unless specified.

\section{A Bundle Coordinate System along the soliton manifold}\label{bundlecoordinates}
\subsection{Linear Decomposition and local coordinates near solitary waves }
Define the soliton manifold consisting of translations of all solitons of \eqref{gkdv} as 
\begin{equation}
\mathcal{M}=\{Q_{c}(\cdot+y): c\in \R^+, y\in \mathbb{R}\}. 
\end{equation} 

To study the dynamics near the travel wave with traveling speed $c>0$, we rewrite \eqref{gkdv} in the traveling frame by letting $u(t,x)=U(t,x-ct)$ which satisfies
\begin{equation}\label{eqintf}
U_{t}-cU_{x}+(U_{xx}+U^{k})_{x}=0.
\end{equation}
For any $y \in \BR$, $Q_c(\cdot +y)$ becomes an equilibrium of \eqref{eqintf}. Linearizing \eqref{eqintf} at $Q_{c}(\cdot+y)$, one has
\begin{equation}\label{le}
U_{t}=JL_{c,y}U,
\end{equation}
where 
\begin{equation*}
J=\partial_{x}, \hspace{0.1 in}  L_{c,y}=c-\partial_{xx}-kQ_{c}^{k-1}(\cdot+y) = (cP + E)'' \big(Q_c(\cdot +y)\big) \in \CL(H^1, H^{-1}). 
 \end{equation*}
For convenience, we let $L_{c}:=L_{c,0}$.  Up to a scalar multiplication, $JL_{c}$ are  conjugate to each other for different $c>0$, through the rescaling 
\begin{equation}\label{scaling}
JL_{c}\CT_0^{\sqrt{c}}U= c^{\frac{3}{2}}\CT_0^{\sqrt{c}}JL_{1}U, \; \text{ where } \; (\CT_0^{\lambda}U)(x) = \lambda^{\frac 2{k-1}} U(\lambda x), 
\end{equation}
and $L_{c,y}$ is conjugate to $L_{c}$ through the translation 
\begin{equation}\label{translation}
L_{c,y}U= \left(L_{c}U(\cdot-y)\right)(\cdot+y).
\end{equation}

\begin{lemma}\label{de1}
For any $c>0$, there exists closed subspaces $X_{c}^{T,e,+,-}$ such that 
\begin{enumerate}
\item $H^{1}= X_{c}^{T}\oplus X_{c}^{e}\oplus X_{c}^{+}\oplus X_{c}^{-}$associated with bounded projection $\Pi_{c}^{T,e,+,-}$;
\item $X_{c}^{T}=\ker L_c =span\{\partial_{x}Q_{c}\}$;
\item $X_{c}^{\pm}=span\{V_{c}^{\pm}\}$, with 

$JL_{c}V_{c}^{\pm}=\pm \lambda_{c}^{\pm}V_{c}^{\pm}$ with $\lambda_{c}= c^{\frac 32} \lambda_1>0$. Moreover $V_c^\pm \in C^\infty$ and, for any $l \ge 0$, $\p_x^l V_c^\pm \to 0$ exponentially as $|x| \to \infty$; 
\item $\p_c Q_c \in X_c^e$ and there exists $A_{c}>0$ such that $\langle L_{c}V^{e},V^{e}\rangle\ge A_{c}\Vert V^{e}\Vert^{2} _{H^{1}(\R)}$ for any $V^{e}\in X_{c}^{e}$. 
\item In this decomposition, $L_c$ and $JL_{c}$ take the following forms 
\begin{equation}\label{E:JL-block-0}
L_c \longleftrightarrow \begin{pmatrix} 0 &0 &0&0\\0&L_c^{e}&0&0\\0&0&0&1\\0&0&1&0\end{pmatrix}, \qquad 
JL_{c}\longleftrightarrow \begin{pmatrix} 0 &A_{Te}&0&0\\0&A_{e}&0&0\\0&0&\lambda_c &0\\0&0&0&- \lambda_c\end{pmatrix},
\end{equation}
where
\[
L_c^e = (\Pi_c^e)^* L_c \Pi_c^e, \quad A_{e}=\Pi_{c}^{e}JL_{c}\Pi_{c}^{e}, \quad  A_{Te}= \Pi_{c}^{T}JL_{c}\Pi_{c}^{e}. 
\]
\end{enumerate}
\end{lemma}

\begin{proof}
In \cite{PW2}, it was shown that $\ker L_c = span\{\p_x Q_c\}$ and all spectra of $JL_{c}$ belong to $i\R$ except one algebraically  simple positive eigenvalue $\lambda_{c}$ and one algebraically simple negative eigenvalue $-\lambda_{c}$ with corresponding eigenfunctions denoted by $V_{c}^{+}$ and $V_{c}^{-}$. Moreover 
\begin{equation} \label{E:temp-1}
\langle L_c V_c^+, V_c^+\rangle = \langle L_c V_c^-, V_c^-\rangle=0, \quad  \langle L_c V_c^+, V_c^-\rangle =1, 
\end{equation}
where $\langle L_c V_c^+, V_c^-\rangle =1$ follows from a possible simple rescaling. 

Since $span\{V_c^+, V_c^-\}$ is invariant under $JL_c$, it is easy to verify directly that its $L_c$-orthogonal complement  
$$Y:=\{v\in H^{1}: \langle L_{c}V_{c}^{+}, v \rangle =\langle L_{c}V_{c}^{-}, v \rangle =0\}$$
is also invariant under $JL_c$. Moreover, \eqref{E:temp-1} implies $\langle L_c \cdot, \cdot \rangle$ is non-degenerate on $span\{V_c^+, V_c^-\}$ and thus $H^1 = span\{V_c^+, V_c^-\} \oplus Y$. Clearly, $X_{c}^{T}\subset Y$. Let 
$$X_c^e=\{v\in Y: \langle \partial_{x}Q_{c}, v \rangle=0\}.$$ 
The block form \eqref{E:JL-block-0} follows directly from the definition of the subspaces. 

In the next we give the explicit forms of the associated projection operators. Any $V\in H^{1}$ can be decomposed as 
\begin{equation}\label{decompv}
V=a^{+}V_{c}^{+} + a^{-}V_{c}^{-} + a^{T}\partial_{x}Q_{c}+V^{e}, 
\end{equation}
where $V^{e}\in X_{c}^{e}$. 
Applying $L_{c}V_{c}^{-}$ and $L_{c}V_{c}^{+}$ to \eqref{decompv}, respectively, we obtain
$$a^{+}=\langle L_{c}V_{c}^{-}, V\rangle,
\quad a^{-}=\langle L_{c}V_{c}^{+}, V\rangle.
$$
Applying $\partial_{x}Q_{c}$ to \eqref{decompv}, we have 
$$a^{T}=\Vert \partial_{x}Q_{c}\Vert_{L^2}^{-2}\left(\langle \partial_{x}Q_{c}, V\rangle- a^{+}\langle \partial_{x}Q_{c}, V_{c}^{+}\rangle-a^{-}\langle\partial_{x}Q_{c}, V_{c}^{-}\rangle\right).$$
Clearly 
\[
\Pi_c^T V = a^T \p_x Q_c, \quad \Pi_c^\pm V = a^\pm V_c^\pm, \quad \Pi_c^e = I - \Pi_c^T - \Pi_c^+ - \Pi_c^-. 
\]
As $\p_x Q_c \in D(J^*) = D(J)$, clearly $A_{Te}= \Pi_{c}^{T}JL_{c}\Pi_{c}^{e}$ is bounded. 

Since, for any $c>0$, $Q_c$ satisfies 
\[
(cP+E)'(Q_c) =0 \Longrightarrow J L_c \p_c Q_c = - J P'(Q_c) = - \p_x Q_c, 
\]
we have 
\[
\langle L_{c}V_{c}^\pm, \p_c Q_c \rangle = \pm \lambda_c^{-1} \langle L_{c} J L_c V_{c}^\pm, \p_c Q_c \rangle =\pm \lambda_c^{-1} \langle L_c V_{c}^\pm, \p_x Q_c \rangle =0
\]
and thus $\p_c Q_c \in Y$. Due to the evenness of $Q_{c}$, it is clear that $\langle \partial_{x}Q_{c}, \partial_{c}Q_{c} \rangle=0$, which implies $\partial_{c}Q_{c}\in X_c^e$. 

To complete the proof of the lemma, we show that uniform positivity of $\langle L_c^e \cdot, \cdot \rangle$. As $L_c$ is a relatively compact perturbation to the uniformly positive operator $c- \p_{xx}$ on $H^1$, it is uniformly positive except in possibly finite many directions. Since $\ker L_c = span \{\p_x Q_c\}$ and $\p_x Q_c$ changes sign exactly once, $L_c$ has only 1-dim negative direction and 1-dim kernel. From \eqref{E:temp-1}, $L_c$ has one negative and one positive directions in $span\{V_c^+, V_c^-\}$ and \[H^1 = span\{V_c^+, V_c^-\} \oplus \ker L_c \oplus X_c^e\] is an $L_c$-orthogonal decomposition, therefore there exists $A_c>0$ such that
\[\langle L_cV^e, V^e\rangle \ge A_{c}\Vert V^{e}\Vert^{2} _{H^{1}(\R)}\] for any $V^e\in X_c^e$. 
This is a special and rather explicit case of the general problems studied in \cite{LZ}.  
\end{proof}

For any $y\in \mathbb{R}$ and $\alpha\in \{T,e,+,-\}$, define $$X_{c,y}^{\alpha}=\{v\in H^{1}\vert v(\cdot-y)\in X_{c}^{\alpha}\}.$$ Clearly, 
$$H^{1}= X_{c,y}^{T}\oplus X_{c,y}^{e}\oplus X_{c,y}^{+}\oplus X_{c,y}^{-}.$$

\begin{lemma}\label{decompvb}
\begin{enumerate}
\item $X_{c,y}^{T}=span \{\partial_{x}Q_{c}(\cdot+y)\}=\ker JL_{c,y}$.
\item  $X_{c,y}^{\pm}=span\{V_{c}^{\pm}(\cdot+y)\}$ and $JL_{c,y}V_{c}^{\pm}(\cdot+y)=\pm \lambda_{c} V_{c}^{\pm}(\cdot+y)$.
\item $\p_c Q_c(\cdot +y) \in X_{c, y}^e$ and 
$\langle L_{c,y}V^{e},V^{e}\rangle\ge A_{c}\Vert V^{e}\Vert^{2} _{H^{1}(\R)}$ for any $V^{e}\in X_{c,y}^{e}$.
\item The associated bounded projection operators $\Pi_{c,y}^{\alpha}$ are smooth in $c,y$ for $\alpha=+,-,T,e$.
\item In the decomposition $H^{1}= X_{c,y}^{T}\oplus X_{c,y}^{e}\oplus X_{c,y}^{+}\oplus X_{c,y}^{-}$, $L_{c, y}$ and $JL_{c,y}$ take the form
\begin{equation}\label{JL-block}
L_{c, y} \longleftrightarrow \begin{pmatrix}  0 & 0 &0&0\\0&L_{c, y}^{e}&0&0\\0&0&0&1\\ 0&0&1&0\end{pmatrix}, \quad 
JL_{c,y}\longleftrightarrow \begin{pmatrix} 0 &A_{Te}(y)&0&0\\0&A_{e}(y)&0&0\\0&0&\lambda_c&0\\0&0&0&-\lambda_c\end{pmatrix}. 
\end{equation}
where
\[
L_{c, y}^e = (\Pi_{c,y}^e)^* L_{c, y} \Pi_{c, y}^e, \quad A_{e}(y) =\Pi_{c, y}^{e}JL_{c, y}\Pi_{c, y}^{e}, \quad  A_{Te}(y) = \Pi_{c, y}^{T}JL_{c, y}\Pi_{c, y}^{e}. 
\]
\item All above blocks are translation invariant in the sense of \eqref{translation}. Moreover, for any $k_1, k_2 \ge 0$ with $L_{c, y}^e + \p_{xx} \in \CL(H^1)$ and $A_{Te}(y), A_e(y) +\p_{xxx}  \in \CL(H^1, L^2)$ depend on $c>0$ and $y\in \R$ smoothly.  
\end{enumerate}
\end{lemma} 
\begin{proof}
All the statements in the lemma, except the smoothness of the operators in $c$ and $y$,  
follow from the translation invariance \eqref{translation} of $JL_{c,y}$. 
To show the smooth of $\Pi_{c, y}^\alpha$ in $c$ and $y$, we use their explicit forms. Any $V\in H^{1}$ can be written as 
\begin{equation}\label{decompv-1}
V=a^{+}V_{c}^{+}(\cdot+y) + a^{-}V_{c}^{-}(\cdot+y) + a^{T}\partial_{x}Q_{c}(\cdot+y)
+ V^{e}, 
\end{equation}
where $V^{e}(\cdot-y)\in X_{c}^{e}$. One can calculate that 
\begin{equation} \label{E:apm}
a^{+}=\langle L_{c,y}V_{c}^{-}(\cdot+y), V\rangle, \quad a^{-}=\langle L_{c,y}V_{c}^{+}(\cdot+y), V\rangle,
\end{equation}
and
\begin{equation} \label{E:aT}
a^{T}= \Vert \partial_{x}Q_{c}\Vert_{L^2}^{-2}
\left(\langle \partial_{x}Q_{c}(\cdot+y), V\rangle- a^{+}\langle \partial_{x}Q_{c}, V_{c}^{+}\rangle-a^{-}\langle \partial_{x}Q_{c}, V_{c}^{-}\rangle\right).
\end{equation}
Therefore 
\begin{equation} \label{E:Pi}
\Pi_{c,y}^T V = a^T \p_x Q_c (\cdot +y), \quad \Pi_{c, y}^\pm V = a^\pm V_c^\pm (\cdot +y)
\end{equation} 
and $\Pi_{c, y}^e = I- \Pi_{c, y}^T - \Pi_{c, y}^+- \Pi_{c, y}^-$. The above explicitly forms yield the smoothness of $\Pi_{c,y}^{+,-,T, e}$ in $c, y$. Finally the smoothness of $L_{c, y}^e + \p_{xx}, A_{Te}(y), A_e(y) +\p_{xxx}$  follow from similar calculations based on the regularity of $Q_c$ and the eigenfunctions $V_c^\pm$. 
\end{proof}

\subsection{A local bundle coordinate system}
In this section, we set up the bundle coordinates near $\mathcal{M}$ precisely and discuss its smoothness.  This subsection is in the same spirit as Section 2.2 in \cite{JLZ}. 

Fixing $c>0$, define a vector bundle $\CX_c^e$ over $ \R$ with fibers $X_{c,y}^{e}$ as
\begin{equation} \label{E:CX}
\CX_c^e = \{(y, V^e) \mid  \; y\in \R, \; V^e \in X_{c, y}^e \},  
\end{equation} 
and balls on this bundle
\begin{equation} \label{E:CX-delta}
\CX_c^e (\delta) = \{ (y, V) \in \CX_c^e  \mid \Vert V \Vert_{H_1} < \delta
\}. 
\end{equation}

Let $ y_\ast \in \R$ and $0<\delta \ll1$, the map
\[
\begin{array}{ccc}
(-\delta,\delta) \times X_{c, y_\ast}^e & \longrightarrow & \CX_c^e\\
(y,V) & \longrightarrow & (y, \Pi_{c, y}^e V)
\end{array}
\]
gives a smooth local trivialization of $\CX_c^e$, where the smoothness is due to the smoothness of $\Pi_{c, y}^e$ with respect to $c$ and $y$. Thus it provides $\CX_c^e$ with a local coordinate system. 

With other subspaces like $X_{c, y}^{T, +, -}$, we will often consider bundles $\R^n\oplus \CX_c^e$ over $\R$ with fibers $\R^n \oplus X_{c, y}^e$, as well as their balls 
\begin{equation} \label{E:Bball} 
B^n (\delta_1) \oplus \CX_c^e (\delta_2) = \{ (y, a, V^e) \mid  a \in \R^n, \ |a| < \delta_1, \  (y, V^e) \in \CX_c^e(\delta_2) \}. 
\end{equation} 

Define an embedding 
\[
Em: \R^{3}\oplus \CX_c^e \to H^{1}
\]
as 
\begin{equation} \label{E:em} \begin{split}
& Em (y, a^T, a^+, a^-, V^e) \\
=&a^{T}\partial_{x}Q_{c}(\cdot+y) + a^{+}V^{+}_{c}(\cdot+y) +a^{-}V^{-}_{c}(\cdot+y)+ V^e \\
= & \left(a^{T}\partial_{x}Q_{c} + a^{+}V^{+}_{c} +a^{-}V^{-}_{c}\right)(\cdot + y) + V^e.
\end{split}\end{equation} 

The embedding $Em^\perp : \R^{2} \oplus \CX_c^e\rightarrow H^1$ defined on the transversal (to the translational direction) bundle will be used in the rest of this paper,
\begin{equation} \label{E:em-perp}
Em^\perp (y, a^+, a^-, V^e) = Em (y,0,a^+, a^-, V^e). 
\end{equation}

Clearly $Em^\perp$ is translation invariant in the sense, for any $\tilde y \in \R$, 
\begin{equation} \label{E:trans-inv-Em} \begin{split}
Em^\perp \big(y+ \tilde y,  a^+, a^-, V^e(\cdot +\tilde y)&\big) 
 = Em^\perp (y, a^+, a^-, V^e)(\cdot +\tilde y).
 \end{split}\end{equation}

On the one hand, according to the above trivialization, given any Banach space $Z$, a mapping $f:Z\to \CX_c^e$ is said to be smooth near some $z_0 \in Z$ if $y(z)$ and $V^e(z) \in X_{c(z_0), y(z_0)}^e$ are smooth in $z$ near $z_0$, where $f(z) = \big(y(z), \Pi^e_{c(z), y(z)} V^e(z)\big)$. Due to the smoothness of $\Pi_{c, y}^e$, in fact this is equivalent to the smoothness of $y(z)$ and $V(z) \in H^1$ where $f(z) = \big( y(z), V(z)\big)$.

On the other hand, for any Banach space $Y$, a mapping $g:\CX_c^e \to Y$ is said to be smooth near some $(y_\ast, V_\ast)$ if 
\[
\tilde g(y, V)= g(y, \Pi_{c, y}^e V), \quad y\in \R, \; V \in X_{c_\ast, y_\ast}^e
\]
is smooth in $(y, V) \in \R \times X_{c, y_\ast}^e$ near $(y_\ast, V_\ast)$. It is straight forward to verify 
\begin{itemize} 
\item $g$ is smooth if and only if locally $g(y, \Pi_{c, y}^e V)$, $y\in \R$, $V \in H^{1}$, is smooth on $\R \times H^1$.
\item $g$ is smooth if and only if locally it is the restriction to $\CX_c^e$ of a smooth mapping defined on $\R \times H^1$;
\item $g$ is smooth if and only if $g\circ f$ is smooth for any smooth $f: Z \to \CX_c^e$ defined on any Banach space $Z$; 
\item $Em$ is smooth with respect to $(y, V^e)$.
\end{itemize}

Near the 2-dim manifold  $\mathcal{M}$ of solitary waves, we will work through the mapping $\Phi$ defined on $ \R^{2} \oplus \CX_c^e$ 
\begin{equation} \label{E:coord-1} \begin{split}
U =  &\Phi (y, a, V^e) = Q_c( \cdot +y) + Em^\perp (y, a, V^e).
\end{split}\end{equation}
For any fixed $c>0$, $\Phi(\cdot)$ is diffeomorphic when $a \in \R^2$ and $\Vert V^e \Vert_{H^1}$ are sufficiently small. 

\begin{remark} \label{R:metric}
Since $\Phi$ is a local diffeomorphism with properties uniform in $y$, locally the total $|a|+\Vert V^e\Vert_{H^1}$ of the transversal components is equivalent to the $H^1$ distance from $\Phi (y, a, V^e)$ to $\CM_c$. 
\end{remark}

This is a smooth vector bundle coordinate system in a neighborhood of $\mathcal{M} \subset H^1$. 
From \eqref{E:em} and \eqref{E:em-perp}, $\Phi$ can be naturally extended into a smooth mapping on $\R^{3} \oplus H^1$. 

\begin{remark} \label{R:coord} 
It is tempting to use the coordinate system 
\[
U = \left(\CT_0^{\sqrt{c}}(Q + a^{+}V_{1}^{+}+ a^{-}V_{1}^{-}+V^{e})\right) (\cdot +y)\
\]
where 
$V^{e}\in X_1^e$ and $y \in \R$. 
However, such rescaling and translation parametrization is not  smooth in $H^1$ because the differentiation in $c$ and $y$ causes a loss of one order regularity in $D_y (\CT^{\sqrt{c}}V^{e})(\cdot +y)$ and $D_c (\CT^{\sqrt{c}}V^{e})(\cdot +y)$. This is one of the main issues in Nakanishi and Schlag \cite{NS1}, where the authors constructed the center-stable manifolds of the manifold of ground states for the Klein-Gordon equation. They introduced a  nonlinear ``mobile distance' to overcome that difficulty. Instead, the above bundle coordinate system \eqref{E:coord-1}, where $V^e \in X_{c, y}^e$ is not directly parametrized by a translation in $y$ and a rescaling in $c$, represents a somewhat different framework based on the observation that, while the parametrization by the spatial translation of $y$ and rescaling of $c$ are not smooth in $H^1$ with respect to $y$ and $c$ respectively,  the vector bundles $X_{c, y}^{T,e, +, -}$ over $\mathcal{M}$ are smooth in $c$ and $y$ as given in Lemma \ref{decompvb}. This geometric bundle coordinate system has been used in \cite{BLZ08, JLZ}, in the latter of which we construct local invariant manifolds near unstable traveling waves of the Gross-Pitaevskii equation. 
 \end{remark}

\subsection{An equivalent form of the gKDV equation near $\mathcal{M}$}\label{derivationofeqns}
Fix $c>0$. Let $U(t,x)$ be any solution to \eqref{eqintf}. If $U(t,x)$ stays in a small neighborhood of $\{ Q_c(\cdot +y) \mid y \in \R\}$, we can use the coordinate system \eqref{E:coord-1} to write it as 
\begin{equation}\label{decomp}
U(t)= \Phi(y(t),a^+(t),a^-(t), V^{e}(t)),
\end{equation}
where $(y(t),a^+(t),a^-(t), V^{e}(t))\in B^2(\delta)\oplus \CX_c^e(\delta)$. 

Plugging \eqref{decomp}, \eqref{E:coord-1} into \eqref{eqintf}, we obtain
\begin{equation}\label{eqdecomp}
\begin{split}
&\partial_{x}Q_{c}(\cdot+y)\partial_{t}y+(\partial_{t}a^{\pm})V^{\pm}_{c}(\cdot+y)+a^{\pm}\partial_{x}V^{\pm}_{c}(\cdot+y)\partial_{t}y+\partial_{t}V^{e}\\
=&(\pm \lambda_{c})a^{\pm} V^{\pm}_{c} (\cdot +y) +JL_{c,y}V^{e}+G(y,a^{+}, a^{-}, V^{e}),
\end{split}
\end{equation}
where
\begin{equation} \label{E:G-1}
\begin{split}
G(y,a^{+}, a^{-},& V^{e})=-\partial_{x}\Big[\left(Q_{c}(\cdot+y)+Em^\perp(c,y,a^+,a^-,V^{e})\right)^{k}\\
&-Q_{c}^{k}(\cdot+y)-kQ_{c}^{k-1}(\cdot+y)Em^\perp(c,y,a^+,a^-,V^{e})\Big] \\
& \quad \; := \p_x \big(G_1(c, y,a^{+}, a^{-}, V^{e})\big).
\end{split}
\end{equation}
Throughout the paper, we often omit the dependence of $G$ and other quantities on $c$ which is mostly fixed. 
As a convention of notations, $a^\pm V_c^\pm$ always means summation of the terms corresponding to `+' and `-' signs. 

We shall apply projections $\Pi_{c, y}^{T, \pm, e}$, by using \eqref{E:apm} and \eqref{E:aT}, to \eqref{eqdecomp} to obtain equations of each components $y, a^\pm, V^e$. 
Firstly applying  $L_{c,y}V^{-}_{c}(\cdot+y)$ to \eqref{eqdecomp}, we obtain
\begin{equation}\label{ap1}
\partial_{t}a^{+}= \lambda_{c} a^{+}+ A^{+}(c,y,a^{+}, a^{-}, V^{e})\partial_{t}y+\bar{G}^{+}(y,a^{+}, a^{-}, V^{e}),
\end{equation}
where
\[
\begin{split}
A^{+}(y,a^{+}, a^{-}, V^{e})=
&-a^{\pm}\langle L_{c,y}V^{-}_{c},\partial_{x}V^{\pm}_{c}\rangle\\&+
\langle (\p_{y}L_{c,y})V^{-}_{c}(\cdot+y)+ L_{c,y}\partial_{x}V^{-}_{c}(\cdot+y), V^{e} \rangle, 
\end{split}
\]
and 
\[
\bar{G}^{+}(y,a^{+}, a^{-}, V^{e})= 
-\langle \p_x L_{c,y}V^{-}_{c}(\cdot +y), G_1(c,y,a^{+}, a^{-}, V^{e}) \rangle.
\]
Similarly, applying  $L_{c,y}V^{+}_{c}(\cdot+y)$ to \eqref{eqdecomp}, we obtain
\begin{equation}\label{am1}
\partial_{t}a^{-}= - \lambda_{c}a^{-}+ A^{-}(c,y,a^{+}, a^{-}, V^{e})\partial_{t}y+\bar{G}^{-}(y,a^{+}, a^{-}, V^{e}),
\end{equation}
where 
\[
\begin{split}
A^{-}(y,a^{+}, a^{-}, V^{e})=
&-a^{\pm}\langle L_{c,y}V^{+}_{c},\partial_{x}V^{\pm}_{c}\rangle\\
&+\langle (\p_{y}L_{c,y})V^{+}_{c}(\cdot+y)+ L_{c,y}\partial_{x}V^{+}_{c}(\cdot+y), V^{e}\rangle,
\end{split}
\]and 
\[
\bar{G}^{-}(y,a^{+}, a^{-}, V^{e})= 
- \langle \p_x L_{c,y}V^{+}_{c}(\cdot +y), G_1(c,y,a^{+}, a^{-}, V^{e})\rangle.
\]

Taking the $L^2$ inner product of \eqref{eqdecomp} with $\partial_{x}Q_{c}(\cdot+y)$, then plugging in \eqref{ap1} and \eqref{am1}, we obtain
\[
\begin{split}
A^{T} (&y,a^{+}, a^{-}, V^{e})\partial_{t}y=- \langle L_{c,y}J \partial_{x}Q_{c}(\cdot+y), V^{e}\rangle\\
& +\langle \partial_{x}Q_{c}(\cdot+y), G(y,a^{+}, a^{-}, V^{e})\rangle -\langle \partial_{x}Q_{c}, V^{\pm}_{c}\rangle \bar{G}^{\pm}(y,a^{+}, a^{-}, V^{e}).
\end{split}
\]
where
\[
\begin{split}
A^{T}(y,a^{+}, & a^{-}, V^{e})=\Vert \partial_{x}Q_{c}\Vert_{L^2}^2+a^{\pm}\langle \partial_{x}Q_{c},\partial_{x}V^{\pm}_{c}\rangle\\
&-\langle V^{e}, \partial^{2}_{x}Q_{c}(\cdot+y)\rangle+ \langle \partial_{x}Q_{c}, V^{\pm}_{c}\rangle A^{\pm}(y,a^{+}, a^{-}, V^{e}).
\end{split}
\]
It is clear that $A^{T}(y,a^{+}, a^{-}, V^{e})> 0$ when $|a^{\pm}|, \Vert V^{e}\Vert _{H^1}\ll 1$, therefore
\begin{equation}\label{eqyp}
\begin{split}
\partial_{t}y= & (A^T)^{-1} 
\big[- \langle L_{c,y} J \partial_{x}Q_{c}(\cdot+y), V^{e}\rangle+\langle \partial_{x}Q_{c}(\cdot+y), G
\rangle\\
&-\langle \partial_{x}Q_{c}, V^{\pm}_{c}\rangle \bar{G}^{\pm}, 
\big]\\
:=& \bar{G}^{T}(y,a^{+}, a^{-}, V^{e})\\
:= &- \Vert \partial_{x}Q_{c} \Vert_{L^2}^{-2}
\langle L_{c,y} J \partial_{x}Q_{c}(\cdot+y), V^{e}\rangle+ G^{T}(y,a^{+}, a^{-}, V^{e}),
\end{split}
\end{equation}
where in the last line we separated terms which are linear and of higher order in $a^\pm$ and $V^e$. Substituting \eqref{eqyp} into \eqref{ap1} and \eqref{am1}, we obtain
\begin{equation}\label{apm}
\partial_{t}a^{\pm}= \pm \lambda_{c} a^{\pm}+ G^{\pm}(y,a^{+}, a^{-}, V^{e}),
\end{equation}
where
\[
G^{\pm}(y,a^{+}, a^{-}, V^{e})=(A^{\pm}\bar{G}^{T}+\bar{G}^{\pm})(y,a^{+}, a^{-}, V^{e}).
\]

Using the higher order regularity of $V_{c}^{\pm}$, $\partial_{x}Q_{c}$ and $\partial_{c}Q_{c}$, one can check that $G^{+,-,T}$ are well-defined and smooth in the energy space and at least quadratic in $a^{\pm}$ and $V^{e}$. 

Applying $\Pi_{c,y}^{e}$ to \eqref{eqdecomp}, we have
\begin{equation}\label{eqve}
\Pi_{c,y}^{e}\partial_{t}V^{e}=\Pi_{c,y}^{e}JL_{c,y}V^{e}+G^{e}(y,a^{+}, a^{-}, V^{e}), 
\end{equation}
where
\[
\begin{split}
G^{e}
=&\Pi_{c,y}^{e}G
-a^{\pm}\bar{G}^{T}
\Pi_{c,y}^{e}\left(\partial_{x}V_{c}^{\pm}(\cdot+y)\right), \quad (I-\Pi_{c, y}^e) G^e=0. 
\end{split}
\]
{\bf An equivalent form of the $V^e$ equation.}  To avoid estimating the geometric equation \eqref{eqve} involving bundles, we first transform it to an equivalent form which may be posed in the whole space $H^1$. Let
\begin{equation}
 \Pi^{\perp}_{c,y}= I-\Pi^{e}_{c,y} \hspace{0.1in}\text{and}\hspace{0.1in}  X^{\perp}_{c,y}= \Pi^{\perp}_{c,y} H^{1}.
\end{equation}
Since  $\Pi^{\perp}_{c,y(t)}V^e(t)=0$ for all $t$, 
differentiating this identity with respect to $t$ yields 
$$\Pi^{\perp}_{c,y}\partial_{t}V^e=\p_t y \p_{y}\Pi_{c,y}^{e} V^e. $$ 
The term $\p_{y}\Pi_{c,y}^{e}V^e$ actually serves as the second fundamental form of the bundle $\CX_c^e$. In order to make the $V^e$ equation posed in the whole space $H^{1}$,  we define a bounded linear map $\mathcal{F}(c,y)\in \CL( H^{1})$

\begin{equation} \label{E:CF}
\mathcal{F}(c, y)V =\p_{y}\Pi_{c,y}^e(\Pi^{e}_{c,y}V-\Pi^{\perp}_{c,y}V) = -\p_y \Pi_{c,y}^\perp (I -2\Pi^{\perp}_{c,y}) V.
\end{equation}
The above form of $\mathcal{F}$, which is a modification of the second fundamental form of $\CX_c^e$, would bring us certain convenience to carry out some calculations in later sections. 

Accordingly, we consider the following extension of \eqref{eqve} 
\begin{equation}\label{eqve2}
\partial_{t}V= \Pi^{e}_{c,y}JL_{c,y}\Pi_{c,y}^{e}V+\p_t y\mathcal{F}(c,y)V +G^e. 
\end{equation}

In the below, we demonstrate that, 
if $V(s)\in X_{c,y(s)}^e$ for some $s$, then $V(t)\in X_{c,y(t)}^e$ for any $t$, and consequently \eqref{eqve} and \eqref{eqve2} are identical according to the definition of $\mathcal{F}(c,y)$. In fact, let $V(t)$ be the solution to
\begin{equation}\label{lv}
 \partial_{t}V = \Pi^{e}_{c,y}JL_{c,y}\Pi_{c,y}^{e}V+ \p_t y \mathcal{F}(c,y)V + f^e(t), \quad f^e(t) \in X_{c, y}^e.
\end{equation}
Since $\Pi_{c,y}^{e}\Pi_{c,y}^{e}=\Pi_{c,y}^{e}$, differentiating this identity in $y$ we have 
\begin{equation} \label{E:PiPi}
\p_{y}\Pi_{c,y}^{e} \Pi_{c,y}^{e}+ \Pi_{c,y}^{e}\p_{y}\Pi_{c,y}^{e}  =\p_{y}\Pi_{c,y}^{e}.
\end{equation}
Using this identity, we calculate
\begin{equation}\label{eqvperp}
\begin{split}
 \partial_{t}(\Pi_{c,y}^{\perp}V)&= \partial_{t}y \Pi_{c,y}^{\perp}\p_{y}\Pi_{c,y}^{e}(\Pi^{e}_{c,y}V-\Pi_{c,y}^{\perp}V) +\partial_{t}y \p_{y}\Pi_{c,y}^{\perp}V\\
  &=-\partial_{t}y\p_{y}\Pi_{c,y}^{\perp} (\Pi^{\perp}_{c,y}V).
 \end{split}
\end{equation}
Observe that the above equation of $\Pi_{c,y}^{\perp}V$ is a well-posed homogeneous linear equation in a finite dimensional space, therefore if $V\left(s\right)\in X^{e}_{c,y\left(s\right)}$, i.e. $\Pi_{c,y\left(s\right)}^{\perp}V\left(s\right)=0$, then $\Pi_{c,y(t)}^{\perp}V(t)=0$ for all $t$.

We will work with \eqref{eqve2} since it is more convenient to obtain estimates compared to \eqref{eqve}. In summary,  in a small neighborhood of solitary waves, we will 
write the gKDV equation in the bundle coordinates \eqref{decomp} as a system consisting of \eqref{eqyp}, \eqref{apm} and \eqref{eqve2}.

\section{Linear Analysis} \label{LinearAnalysis}
The aim of this section is to establish linear estimates to be utilized on equation \eqref{eqve2}. The unknown of \eqref{eqve2} is in $\CX^e$, however,  with our definition of $\CF$ it is also well-posed in $H^1$. As one will see later, the following more general form of \eqref{eqve2} with the unknown $V\in H^1$ (not necessarily in $\CX^e$) will be more convenient for us to use
\begin{equation}\label{eqve3}
\partial_{t}V= \Pi^{e}_{c,y}JL_{c,y}\Pi_{c,y}^eV+\p_t y \mathcal{F}(c,y)V +f(t, x), 
\end{equation}
where $y = y(t)$ is a given Lipschitz function. 

With our definition of $\mathcal{F}$, the equations of 
$V^{e}=\Pi_{c,y}^{e}V$ and $V^{\perp}= \Pi_{c,y}^{\perp}V$ are decoupled. In fact, similar to \eqref{eqvperp}, one can calculate 
\begin{equation}\label{vevp1}
\partial_{t}V^{e}= \Pi^{e}_{c,y}JL_{c,y}V^{e}+\partial_{t}y\mathcal{F}(c,y)V^{e}+f^{e}
\end{equation}
\begin{equation}\label{vevp2}
\partial_{t}V^{\perp}=\partial_{t}y\mathcal{F}(c,y)V^{\perp}+f^{\perp}.
\end{equation}
where  $f^{e}=\Pi_{c,y}^{e}f $ and  $f^{\perp}= \Pi_{c,y}^{\perp}f.$ We will work with \eqref{vevp1}, \eqref{vevp2}, and estimate $V^{e}$ and $V^{\perp}$ separately. In particular, we note that \eqref{E:CF} and \eqref{E:PiPi} imply
\begin{equation} \label{E:CF-1}
\CF(c, y) X_{c, y}^e \subset X_{c, y}^\perp, \quad \CF(c, y) X_{c, y}^\perp \subset X_{c, y}^e.
\end{equation} 

\noindent
{\bf  Energy Estimates of homogeneous Equations.} 
Starting with energy estimates, we analyze \eqref{eqve3} with $f=0$. Fix $c>0$. According to Lemma \ref{decompvb}, there exists $A_c>0$ such that $\langle L_{c,y}V^{e}, V^{e}\rangle \ge A_c \Vert V^{e}\Vert _{H^{1}}$ for any $V^{e}\in X_{c,y}^{e}$, 
therefore $\langle L_{c,y}V^{e}, V^{e}\rangle^{1/2}$ is equivalent to the $H^{1}$ norm on $X_{c, y}^e$.  For any $V\in H^1$, define a semi-norm
\begin{equation}\label{tildeh1}
\Vert V \Vert _{\widetilde{H}_y^{1}}:= \langle L_{c,y} \Pi_{c, y}^e V, \Pi_{c, y}^e V \rangle^{1/2}\sim \Vert \Pi_{c, y}^e V\Vert _{H^{1}}.
\end{equation}
which depends on $c$ and $y$.

\begin{lemma}\label{sh1}
Assume that $f=0$ and $y(t)$ satisfies $\Vert \partial_{t}y\Vert_{L^{\infty}}\le \sigma$, then \eqref{eqve3} generates a bounded evolution operator $$S(t,s)\in \mathcal{L}(H^1, H^1), \qquad \forall s,t\in [t_{0},t_{0}+T]$$ satisfying 
$$S(t,s)\in \mathcal{L}(X_{c,y(s)}^{e}) 
\quad \text{and}\quad S(t,s)\in \mathcal{L}(X_{c,y(s)}^{\perp}). 
$$ 
Moreover, there exists a constant $C$ independent of $y$, $\sigma$, such that for any $V^{e}(s)\in X_{c,y(s)}^{e}$ and $V^{\perp}(s)\in X_{c,y(s)}^{\perp}$, we have 
$$\Vert S(t,s)V^{e}(s)\Vert _{\widetilde{H}_{y(t)}^{1}}\le e^{C\sigma |t-s|}\Vert V^{e}(s)\Vert _{\widetilde{H}_{y(t_0)}^{1}},$$
and
$$\Vert S(t,s)V^{\perp}(s)\Vert _{H^{1}}\le e^{C\sigma |t-s|}\Vert V^{\perp}(s)\Vert _{H^{1}}.$$
\end{lemma}

\begin{proof}
Due to the high regularity of $Q_{c}$ and $X_{c,y}^{\perp}$, $\Pi^{e}_{c,y}JL_{c,y}\Pi_{c,y}^eV$ is a bounded perturbation to $JL_{c,\infty}=\partial_{x}(c-\partial_{xx})$. Moreover, $\mathcal{F}(c,y)\in \mathcal{L}(H^1)$, therefore \eqref{eqve3} is well-posed in $H^1$ and $S(t,s)\in \mathcal{L}(H^1)$ is well-defined. 

The invariance of $S(t,s)$ in the bundles $\CX^{e}$ and $(c,y, X_{c,y}^{\perp})$ is an immediate consequence of the decoupled form  of the equations \eqref{vevp1} and \eqref{vevp2} of $V^{e}$ and $V^{\perp}$. 

It remains to prove the two inequalities. We have 
\begin{equation}
\langle L_{c,y}V^{e}, V^{e}\rangle_{t}= 2\langle L_{c,y}V^{e}, \partial_{t}V^{e}\rangle+ \partial_{t}y\langle (\p_{y}L_{c,y}) V^{e}, V^{e}\rangle.
\end{equation}
One the one hand, clearly from  \eqref{vevp1}, \eqref{E:CF-1}, and the fact that $H^1 = X_{c, y}^e \oplus X_{c, y}^\perp$ is a $L_{c, y}$-orthogonal decomposition, we have 
$$\langle L_{c,y}V^{e}, \partial_{t}V^{e}\rangle=\langle L_{c,y}V^{e}, \Pi_{c, y}^e JL_{c,y}V^{e}
\rangle=\langle L_{c,y}V^{e}, JL_{c,y}V^{e}\rangle=0.$$ 
On the other hand, using the high regularity of $Q_{c}$, it is easy to check that there exists constants $C'$ and $C$ such that 
$$|\partial_{t} y\langle (\p_{y}L_{c,y})V^{e}, V^{e}\rangle|\le C'\sigma \Vert V^{e}\Vert ^{2}_{H^{1}}\le C \sigma \langle L_{c,y}V^{e}, V^{e}\rangle.$$
It follows that 
$$\langle L_{c,y}V^{e}, V^{e}\rangle_{t}\le C\sigma \langle L_{c,y}V^{e}, V^{e}\rangle,$$
which implies the first inequality. 

Taking the $H^1$ inner product of \eqref{vevp2} with $V^{\perp}$, one immediately obtains the second inequality . 
\end{proof}

\begin{remark}
It is worth mentioning that in the above lemma the coefficient in front of $e^{\sigma t}$ is 1, which is crucial in future iteration steps. 
\end{remark}

\noindent
{\bf Smoothing Space-Time Estimates of Homogeneous Equations} In the rest of the section, we establish smoothing space-time  estimates for \eqref{vevp1} based on the space-time estimates established in \cite{KPV} for the Airy equation $u_t+u_{xxx}=0$.   

\begin{lemma} (\cite{KPV})
Let $W(t)$ be the group generated by
\begin{equation} 
u_{t}+u_{xxx}=0.
\end{equation}
The following estimates hold:
\begin{enumerate}
\item If $u_{0}\in L^{2}(\R)$, then 
\begin{equation}\label{hl1}
\Vert \partial_{x}W(t)u_{0}\Vert _{L^{\infty}_{x}L^{2}_{t}}\le C \Vert u_{0}\Vert _{L^{2}},
\end{equation}
and 
\begin{equation}\label{hl2}
\Vert D_{x}^{1/4}W(t)u_{0}\Vert _{L^{4}_{t}L^{\infty}_{x}}\le C \Vert u_{0}\Vert _{L^{2}}.
\end{equation}
\item If $u_{0}\in H^{s}(\R)$ with $s>3/4$, then for any $\rho>3/4$,
\begin{equation}\label{hl3}
\Vert W(t)u_{0}\Vert _{L^{2}_{x}L^{\infty}_{[0,T]}}\le C (1+T)^{\rho}\Vert u_{0}\Vert _{H^{s}}.
\end{equation}
\item If $g(t,x)\in L^{1}_{x}L^{2}_{t}$, then for any $T>0$ (can be $\infty$),
\begin{equation}\label{il1}
\left\Vert \partial_{x}\int_{0}^{t}W(t-s)g(
s)ds\right\Vert _{L^{\infty}_{[0,T]}L^{2}_{x}}\le C \Vert g\Vert _{L^{1}_{x}L^{2}_{[0,T]}},
\end{equation}
and
\begin{equation}\label{il2}
\left\Vert \partial_{xx}\int_{0}^{t}W(t-s)g(s)ds\right\Vert _{L^{\infty}_{x}L^{2}_{[0,T]}}\le C \Vert g\Vert _{L^{1}_{x}L^{2}_{[0, T]}}.
\end{equation}
\end{enumerate}
\end{lemma}

Motivated by the above estimates, define norms $\Vert \cdot\Vert _{ST'_{[t_{0},t_{0}+T]}}$ as
\begin{equation} \begin{split} 
\Vert V\Vert _{ST'_{[t_{0},t_{0}+T]}}=\max\{&\Vert V\Vert _{L^{\infty}_{[t_{0},t_{0}+T]}H_{x}^{1}}, \Vert \partial_{xx}V\Vert _{L^{\infty}_{x}L^{2}_{[t_{0},t_{0}+T]}},\\
&\Vert V\Vert _{L^{2}_{x}L^{\infty}_{[t_{0},t_{0}+T]}},\Vert \partial_{x}V\Vert _{L^{4}_{[t_{0},t_{0}+T]}L^{\infty}_{x}}\},
\end{split} \end{equation}
and 
$\Vert \cdot\Vert _{ST^{c}_{[t_{0},t_{0}+T]}}$ as 
\begin{equation}
\Vert V(t,x)\Vert _{ST^{c}_{[t_{0},t_{0}+T]}}=\Vert V(t,x-ct)\Vert _{ST'_{[t_{0},t_{0}+T]}}.
\end{equation}

\begin{proposition}\label{leve}
There exists \;$C>0$ independent of $y(\cdot)$, $\sigma \le 1$, $t_0$, and $T$, such that  for any $y(\cdot)\in C^{1}([t_{0},t_{0}+T])$ with $\|\partial_{t}y\|_{L^\infty}\le \sigma$ and any $V^{e}(t_{0})\in X^{e}_{c,y(t_{0})}$, we have
\begin{equation} \label{E:leve-1}
\begin{split}
&\Vert S(t,t_{0})V^{e}(t_{0})\Vert _{ST^{c}_{[t_{0},t_{0}+T]}}+\left\Vert\int_{t_{0}}^{t}S(t,s)f^{e}(s)ds\right\Vert _{ST^{c}_{[t_{0},t_{0}+T]}}\\
\le & C (1+T^{4})e^{C\sigma T}\Vert V^{e}(t_{0})\Vert _{H^{1}}+ C\int_{t_{0}}^{t_{0}+T}(1+(t_0+T-s)^4)e^{C\sigma (t_0+T-s)}\Vert f^{e}(s)\Vert _{H_{x}^{1}}ds,
\end{split}
\end{equation}
\begin{equation} \label{E:temp-1.4}
\begin{split}
&\Vert S(t,t_{0})V^{e}(t_{0})\Vert _{\wt{H}_{y(t)}^1}+\left\Vert\int_{t_{0}}^{t}S(t,s)f^{e}(s)ds\right\Vert _{\wt{H}_{y(t)}^1}\\\le &e^{C\sigma (t-t_0)}\Vert V^{e}(t_{0})\Vert _{\wt H_{y(t_0)}^{1}}+C\int_{t_{0}}^{t}e^{C\sigma (t-s)}\Vert f^{e}(s)\Vert _{H_{x}^{1}}ds,
\end{split}
\end{equation}
and 
\begin{equation}
\begin{split}
&\Vert S(t,t_{0})V^{\perp}(t_{0})\Vert _{H^{1}}+ \left\Vert\int_{t_{0}}^{t}S(t,s)f^{\perp}(s)ds\right\Vert _{H_{x}^{1}}\\\le &e^{C\sigma (t-t_0)}\Vert V^{\perp}(t_{0})\Vert _{H^{1}}+\int_{t_{0}}^{t}e^{C\sigma (t-s)}\Vert f^{\perp}(s)\Vert _{H_{x}^{1}}ds.
\end{split}
\end{equation}
\end{proposition}

It is crucial that the coefficient in front of $e^{C\sigma (t-t_0)}\Vert V^{e}(t_{0})\Vert _{\wt H^{1}}$ in \eqref{E:temp-1.4} is 1, which makes an iteration argument possible based on this inequality. 

Our proof is based on perturbative arguments. We split the proof of this proposition into several lemmas. The following technical lemma provides estimates which will be used repeatedly in non-homogeneous estimates throughout this paper. 
\begin{lemma}\label{l2li}
Assuming that $f\in H^1(\R)\cap W^{1,\infty}(\R)$ and $\rho(t)\in C^{1}\left([t_{0},t_{0}+T]\right)$ satisfying $\vert \rho'(t)\vert_{C^{0}([t_{0},t_{0}+T])}\le M$ for some constant $M$, then the following estimates hold

\begin{equation} \label{E:l2li-1} 
\left\Vert f\left(x-\rho(t)\right)\right\Vert_{L^{2}_{x}L^{\infty}_{[t_{0},t_{0}+T]}}\le MT\Vert f'(x)\Vert_{ L^{2}(\R)}+\Vert f(x)\Vert_{ L^{2}(\R)};
\end{equation} 
\begin{equation} \label{E:l2li-2} 
\left\Vert f\left(x-\rho(t)\right)\right\Vert_{L^{\infty}_{x}L^{2}_{[t_{0},t_{0}+T]}}\le MT^{3/2}\Vert f'(x)\Vert_{ L^{\infty}(\R)}+T^{1/2}\Vert f(x)\Vert_{ L^{\infty}(\R)}. \end{equation} 

\end{lemma}

\begin{proof}
Since $$f\left(x-\rho(t)\right)=f\left(x-\rho(t_{0})\right)-\int_{t_{0}}^{t}f'\left(x-\rho(s)\right)\rho'(s)ds,$$ by the Minkowski's integral inequality, we have 
\begin{equation}
\begin{split}
\left\Vert f\left(x-\rho(t)\right)\right\Vert_{L^{2}_{x}L^{\infty}_{[t_{0},t_{0}+T]}}
\le & \Vert f(x)\Vert_{ L^{2}(\R)}+ M\left\Vert\int_{t_{0}}^{t_{0}+T}\left\vert f'\left(x-\rho(s)\right)\right\vert ds\right\Vert_{L_{x}^{2}(\R)}\\
\le & \Vert f(x)\Vert_{ L^{2}(\R)}+ M\int_{t_{0}}^{t_{0}+T}\left\Vert f'\left(x-\rho(s)\right)\right\Vert_{L^{2}_{x}}ds\\
\le & \Vert f(x)\Vert_{ L^{2}(\R)}+MT\Vert f'(x)\Vert_{ L^{2}(\R)}.
\end{split}
\end{equation}
The second inequality can be proved in a similar fashion and we omit the details. 

\end{proof}

\begin{lemma}\label{sthg}
Assume $y(t)$ satisfies \;$|\partial_{t}y(t)|_{L^{\infty}}\le\sigma\le 1$. Let $V^{e}(t)=S(t,t_{0})V^{e}(t_{0})$ and $\widetilde{V}^{e}(t,x)= V^{e}(t,x-ct)$ with $V^{e}(t_{0})\in X_{c,y(t_{0})}^{e}$. Then there exists some constant $C$ independent of $y(\cdot)$, $\sigma$, and $T$, such that

\begin{equation} \label{E:sthg-1}
\Vert \partial_{xx}\widetilde{V}^{e}\Vert _{L_{x}^{\infty}L^{2}_{[t_{0},t_{0}+T]}}\le C (1+T^{3/2}e^{C\sigma T})\Vert V^{e}(t_{0})\Vert _{H_x^{1}}.
\end{equation}

\begin{equation} \label{E:sthg-2}
\Vert \widetilde{V}^{e}\Vert _{L^{2}_{x}L^{\infty}_{[t_{0},t_{0}+T]}}\le C (1+T^{4})e^{C\sigma T}\Vert V^{e}(t_{0})\Vert _{H_x^{1}}.
\end{equation}

\begin{equation} \label{E:sthg-3}
\Vert \partial_{x}\widetilde{V}^{e}\Vert _{L^{4}_{[t_{0},t_{0}+T]}L^{\infty}_{x}}\le C (1+T^{3})e^{C\sigma T}\Vert V^{e}(t_{0})\Vert _{H_x^{1}}.
\end{equation}

\end{lemma}

\begin{proof}
Rewrite \eqref{vevp1} with $f=0$ as
$$\partial_{t}V^{e}= JL_{c,y}V^{e}+\Theta(t)V^{e},$$
where $\Theta(t)V=\p_t y \mathcal{F}\big(c,y(t)\big) V - \Pi_{c,y(t)}^{\perp}JL_{c,y(t)}V$.
Clearly, $\widetilde{V}^{e}(t)$ satisfies
\begin{equation}\label{le2}
\partial_{t}\widetilde{V}^{e}=-\partial^{3}_{x}\widetilde{V}^{e}-\partial_{x}\left(kQ^{k-1}_{c}(\cdot+y(t)-ct)\widetilde{V}^{e}\right)+\widetilde{\Theta}(t)\widetilde{V}^{e}\\
\end{equation}
where 
$$\left(\widetilde{\Theta}(t)\widetilde{V}^{e}\right)(\cdot)= \left(\Theta(t)V^{e}\right)(\cdot-ct)=\left(\Theta(t)\widetilde{V}^{e}(t,\cdot+ct)\right)(\cdot-ct).$$
Using the Duhamel's principle, we write \eqref{le2} as
\begin{equation} \label{E:temp-1.5}  \begin{split}
\widetilde{V}^{e}(t)= W(t-t_{0})\widetilde{V}^{e}(t_{0})+\int_{t_{0}}^{t}&W(t-s)\big[\widetilde{\Theta}(s)\widetilde{V}^{e}(s)\\
&-\partial_{x}\left(kQ^{k-1}_{c}\left(\cdot+y(s)-cs\right)\widetilde{V}^{e}(s)\right)\big]ds.
\end{split}\end{equation}

\noindent $\bullet$ {\bf Proof of \eqref{E:sthg-1}.} By \eqref{hl1}, \eqref{il2}, one immediately has
\begin{equation}
\begin{split}
\Vert &\partial_{xx} \widetilde{V}^{e}(t)\Vert _{L_{x}^{\infty}L^{2}_{[t_{0},t_{0}+T]}} \le C\left\Vert \p_x \left( Q^{k-1}_{c}\big(\cdot+y(t)-ct\big)\widetilde{V}^{e}\right)\right\Vert _{L^{1}_{x}L^{2}_{[t_{0},t_{0}+T]}}\\
& \quad + C\left\Vert \int_{t_{0}}^{t}W(t-s)\partial_{xx}\left(\widetilde{\Theta}(s)\widetilde{V}^{e}(s)\right)ds\right\Vert _{L^{\infty}_{x}L^{2}_{[t_{0},t_{0}+T]}} + C \Vert \widetilde{V}^{e}(t_{0})\Vert _{H^1}.
\end{split}
\end{equation}
Using Lemma \ref{l2li}, one has 
\[
\begin{split}
&\left\Vert \p_x \Big( Q^{k-1}_{c}\big(\cdot+y(t)-ct\big)\widetilde{V}^{e}\Big)\right\Vert _{L^{1}_{x}L^{2}_{[t_{0},t_{0}+T]}}\\
\le &
\Vert Q^{k-1}_{c}(\cdot+y(t)-ct)\Vert _{H_{x}^{1}L^{\infty}_{[t_{0},t_{0}+T]}} 
\Vert \wt V^e\Vert _{H_{x}^1 L^{2}_{[t_{0},t_{0}+T]}}
\le  C T^{1/2}(1+T)\Vert \wt{V}^e\Vert _{L^{\infty}_{[t_{0},t_{0}+T]}H_{x}^1}.
\end{split}
\]
From the spatial regularity and decay of functions in $X_{c,y}^{\perp}$ (of dimension 3) and the expression of $\Pi_{c,y(t)}^{\perp}$ given in Lemma \ref{decompvb}, one can easily check that  
$\Vert \Theta(t)\Vert _{\mathcal{L}(L^{2}, H^l)}\le C$ for any $l \ge 0$. 
Therefore, 
we have 
\begin{equation}\label{E:temp-2}
\Vert \widetilde{\Theta}(t)\widetilde{V}^{e} \Vert_{H^{l}_{x}}= \Vert \Theta(t)V^{e}\Vert_{H^{l}_{x}}\le C \Vert V^{e}\Vert_{L^{2}_{x}}= C \Vert \widetilde{V}^{e}\Vert_{L_{x}^2}.
\end{equation}
From the above and Minkowski's inequality, we obtain
\[
\begin{split}
& \left\Vert \int_{t_{0}}^{t}W(t-s)\partial_{xx}\left(\widetilde{\Theta}(s)\widetilde{V}^{e}(s)\right)ds\right\Vert _{L^{\infty}_{x}L^{2}_{[t_{0},t_{0}+T]}}\\
\le &C 
\left\Vert \int_{t_{0}}^{t}\left\Vert W(t-s)\partial_{xx}\left(\widetilde{\Theta}(s)\widetilde{V}^{e}(s)\right)\right\Vert _{L^{\infty}_{x}}ds \right\Vert_{L_{[t_0, t_0+T]}^2}
\\
\le & C  
\left\Vert \int_{t_{0}}^{t} \Vert \widetilde{\Theta}(s)\widetilde{V}^{e}(s)\Vert _{L_{[t_{0},t_{0}+T]}^{\infty}H^{3}_{x}}ds \right\Vert_{L_{[t_0, t_0+T]}^2}
\le C T^{3/2}\Vert \widetilde{V}^{e}\Vert _{L^{\infty}_{[t_{0},t_{0}+T]}L^{2}_{x}}.
\end{split}
\]
Therefore, we obtain
\[
\Vert \partial_{xx}\widetilde{V}^{e}(t)\Vert _{L_{x}^{\infty}L^{2}_{[t_{0},t_{0}+T]}}\le C \Vert \widetilde{V}^{e}(t_{0})\Vert _{H^1}+CT^{1/2}(1+T)\Vert \widetilde{V}^{e}\Vert _{L_{[t_{0},t_{0}+T]}^{\infty}H_{x}^1}.
\]
Inequality \eqref{E:sthg-1} follows from the above inequality and Lemma \ref{sh1}.

\noindent $\bullet$ {\bf Proof of \eqref{E:sthg-2}.} Using \eqref{E:l2li-1}, \eqref{E:sthg-1}, and Lemma \ref{sh1}, we first obtain 
\begin{equation}\label{ii1} \begin{split}
& \int_{t_{0}}^{t_{0}+T}\left\Vert \partial_{x}\left(kQ^{k-1}_{c}(\cdot+y(s)-cs)\widetilde{V}^{e}(s)\right)\right\Vert _{H^{1}_{x}}ds \le C \Vert \widetilde{V}^{e}\Vert _{L_{[t_0, t_0+T]}^1 H_x^1} \\
 &  \qquad \quad + C T^{\frac 12} \Vert Q^{k-1}_{c}(\cdot+y(t)-ct)\Vert_{L_x^2 L_{[t_0, t_0+T]}^\infty} \Vert \partial_{xx} \widetilde{V}^{e}\Vert _{L_x^\infty L_{[t_0, t_0+T]}^2} \\
\le & C (T^{1/2}+T^{3})e^{C\sigma T}\Vert V^{e}(t_{0})\Vert _{H^{1}}.
\end{split}\end{equation}
Along with \eqref{hl3} and Lemma \ref{sh1}, it implies 
\[\begin{split} 
&\left\Vert \int_{t_0}^{t}W(t-s)\partial_{x}\left(kQ^{k-1}_{c}(\cdot+y(s)-cs)\widetilde{V}^{e}(s)\right)ds\right\Vert _{L^{2}_{x}L^{\infty}_{[t_0, t_0+T]}}\\
\le & C\int_{t_{0}}^{t_{0}+T}\left\Vert W(t-s)\partial_{x}\left(kQ^{k-1}_{c}(\cdot+y(s)-cs)\widetilde{V}^{e}(s)\right)\right\Vert _{L^{2}_{x}L^{\infty}_{t\in[s,t_{0}+T]}}ds \\
\le & C (1+T) \int_{t_{0}}^{t_{0}+T} \left\Vert \partial_{x}\left(kQ^{k-1}_{c}(\cdot+y(s)-cs)\widetilde{V}^{e}(s)\right)\right\Vert _{H^{1}_{x}} ds\\
\le & C (1+T^{4})e^{C\sigma T}\Vert V^{e}(t_{0})\Vert _{H^{1}}.
\end{split}\]
Using \eqref{E:temp-2}, in a similar manner we may obtain 
\[\begin{split} 
&\left\Vert \int_{t_{0}}^{t}W(t-s)\widetilde{\Theta}(s)\widetilde{V}^{e}(s)ds\right\Vert _{L^{2}_{x}L^{\infty}_{[t_0, t_0+T]}} \\
\le& \int_{t_{0}}^{t_{0}+T}\Vert W(t-s)\widetilde{\Theta}(s)\widetilde{V}^{e}(s)\Vert _{L^{2}_{x}L^{\infty}_{t\in[s,t_{0}+T]}}ds \\
\le & C (1+T) \int_{t_{0}}^{t_{0}+T} \Vert \widetilde{V}^{e}(s)\Vert _{L_{x}^{2}} ds \le C (1+ T^2) e^{C\sigma T}\Vert V^{e}(t_{0})\Vert _{H^{1}}.
\end{split}\]
Inequality \eqref{E:sthg-2} follows from \eqref{hl3}, \eqref{E:temp-1.5}, and the above inequalities.

\noindent $\bullet$ {\bf Proof of \eqref{E:sthg-3}}: Using the Minkowski's integral inequality, \eqref{hl2} and the fact that $H^{1}(\R)\subset\dot{H}^{3/4}(\R)$, one can verify
\begin{equation}\label{iii1}
\begin{split}
&\left\Vert \partial_{x}\int_{t_{0}}^{t}W(t-s)g(s)ds\right\Vert _{L^{4}_{[t_{0},t_{0}+T]}L^{\infty}_{x}}\\
\le & \int_{t_{0}}^{t_{0}+T}\left\Vert \partial_{x}W(t-s)g(s)\right\Vert _{L^{4}_{t\in[s,t_{0}+T]}L^{\infty}_{x}}ds
\le C\int_{t_{0}}^{t_{0}+T}\Vert g(s)\Vert _{H^{1}_{x}}ds.
\end{split}
\end{equation}
Along with \eqref{E:temp-1.5}, \eqref{hl2}, \eqref{iii1}, and \eqref{ii1}, it implies \eqref{E:sthg-3}.
\end{proof}

With the above preparation, now we prove Proposition \ref{leve}.

\begin{proof}[{\bf Proof of Proposition \ref{leve}}]
By Lemma \ref{sthg} and the Minkowski's integral inequality, one has
\begin{equation}
\begin{split}
&\left\Vert \left(S(t,t_{0})V^{e}(t_{0})\right)\right\Vert_{ST^{c}_{[t_{0},t_{0}+T]}} +  \left\Vert\int_{t_{0}}^{t}\left(S(t,s)f^{e}(s)\right)ds\right\Vert_{ST^c_{[t_0,t_{0}+T]}}\\
\le &\left\Vert \left(S(t,t_{0})V^{e}(t_{0})\right)\right\Vert_{ST^{c}_{[t_{0},t_{0}+T]}} +  \int_{t_{0}}^{t_{0}+T}\left\Vert S(t,s)f^{e}(s)\right\Vert_{ST^c_{t\in [s,t_{0}+T]}}ds
\end{split}
\end{equation}
and thus \eqref{E:leve-1} follows. 
 The other two can be obtained directly from Lemma \ref{sh1}. 
\end{proof}

To end this section, we estimate the difference between solutions to \eqref{eqve3} along base paths $y_i(t)$ and with non-homogeneous terms $f_i(t,x)$, $i=1,2$. We have 

\begin{lemma}\label{diff-hom}
Assume $y_i(t)$ satisfies $\Vert \p_t y_i\Vert_{L^{\infty}_{[t_0,t_0+T]}}\le \sigma\le1$, $i=1,2$. There exists a constant $C>0$ independent of $T>0, y_i(\cdot),\sigma, t_0$, and the non-homogeneous terms $f_i(t,x)$ such that,  for $V_{01}, V_{02} \in H^1$, the solutions $V_i(t) = S_i(t,s) V_{0i}$ of \eqref{eqve3}  along the paths $y_i(t)$ with non-homogeneous terms $f_i(t,x)$ satisfy, \begin{equation*}
\begin{split}
\Vert \Pi_{c, y_2}^e (V_2 - V_1)  & 
\Vert_{ST^{c}_{[t_0,t_0+T]}} 
\le  Ce^{C\sigma T}\big((1+T^4)\big(\Vert  \Pi_{c, y_2(t_0)}^e (V_{02} - V_{01})  \Vert_{H^1}+ \Vert \Pi_{c, y_2}^e (f_2 - f_1)\Vert_{L_t^1 H_x^1}\big) 
 \\ 
& +T^{\frac 12}(1+T^5) \Vert y_{1}-y_{2}\Vert_{C^{0,1}} \big(\Vert V_{01}\Vert_{H^1} + \Vert  f_1(t) \Vert_{L_t^1 H_x^1} \big) 
\big),
\end{split}\end{equation*}

\begin{equation*}
\begin{split}
\Vert  V_2(t_0+T)  - V_1 (&t_0+T)  \Vert_{\wt{H}_{y_2(t_0+T)}^1}\le  e^{C\sigma T}\big(\Vert V_{02}- V_{01} \Vert_{\wt{H}_{y_2(t_0)}^1} + C \Vert \Pi_{c, y_2}^e (f_2 - f_1)\Vert_{L_t^1 H_x^1}\\
&
+ CT^{\frac 12} (1+T^5) \Vert y_{1}-y_{2}\Vert_{C^{0,1}}\big(\Vert V_{01}\Vert_{H^1} + \Vert  f_1(t) \Vert_{L_t^1 H_x^1} \big)\big)
\end{split}
\end{equation*}
and for any $l>0$
\begin{equation*}
\begin{split}
\Vert \big( (I-\Pi_{c, y_2}^e)& ( V_2  - V_1) \big)  (t_0+T)
\Vert_{H^l} \le  e^{C\sigma T}\big( \Vert (I-\Pi_{c, y_2(t_0)}^e)(V_{02}  - V_{01}  )
\Vert_{H^1} \\
& + \Vert (I-\Pi_{c, y_2}^e) (f_2 - f_1)\Vert_{L_t^1 H_x^1} + C T \Vert y_{1}-y_{2}\Vert_{C^{0,1}}\big(\Vert V_{01}\Vert_{H^1} + \Vert  f_1(t) \Vert_{L_t^1 H_x^1} \big)\big)
\end{split}
\end{equation*}
where all the norm in $t$ are taken on the interval $[t_0, t_0+T]$. 
\end{lemma}

\begin{remark}
For $T<0$, the estimates in this sections still hold with $T$ replaced by $|T|$. In the case where estimates on $\Pi_{c, y_2}^e V_2 - \Pi_{c, y_1}^e V_1$ is required, it can be obtained by observing 
\[
\Pi_{c, y_2}^e V_2 - \Pi_{c, y_1}^e V_1 = \Pi_{c, y_2}^e ( V_2 - V_1) + O(|y_2 -y_1| \Vert V_1 \Vert).
\]
\end{remark}

\begin{proof} 
Equation \eqref{eqve3} implies 
\begin{equation} \label{E:eqve3-d}
\partial_{t}(V_{1}-V_2)
=\Pi^{e}_{c,y_{2}}JL_{c,y_{2}}\Pi^{e}_{c,y_{2}}(V_{1}-V_2)+\p_t y_2 \mathcal{F}(c,y_2)(V_{1}-V_2)+\Delta_{1}^{2} + f_2 - f_1,
\end{equation}
where
\begin{equation}\label{dh12} \begin{split} 
\Delta_{1}^{2} =& \big(\Pe\J\Pe-\Pee\JJ\Pee+ \p_t y_1\mathcal{F}(c,y_1)  -\p_t y_2 \mathcal{F}(c,y_2)\big) V_{1}\\
=&\Big( \Pe\J(\Pe-\Pee) +\Pe(\J-\JJ)\Pee  \\& + (\Pe-\Pee)\JJ\Pee  + ( \p_t y_1-\p_t y_2 ) \mathcal{F}(c,y_1)\\
&+ \p_t y_2 \big(\mathcal{F}(c,y_1)-  \mathcal{F}(c,y_2) \big)  \Big)V_1.
\end{split}\end{equation}

By Lemma \ref{sh1} and Proposition \ref{leve}, we have
\begin{equation}\label{hgt1}
\begin{split}
\Vert \Pee(&V_{1}-V_2)   \Vert _{ST^{c}_{[t_{0},t]}}
\le CEx(t-t_0)  \Vert \Pi^{e}_{c,y_{2}(t_0)} (V_{01} -V_{02}) \Vert_{H^{1}}\\
&+C\int_{t_{0}}^{t}Ex(t-s) \big( \Vert (\Pee\Delta_{1}^{2}) (s)\Vert _{H_{x}^{1}} + \Vert \Pi_{c, y_2(s)}^e \big(f_2(s) - f_1(s)\big) \Vert_{H_x^1} \big)ds,
\end{split}
\end{equation}
where $Ex(s) = (1+s^{4})e^{C\sigma s}$ and 
\begin{equation}\label{hgt1.5}
\begin{split}
\Vert V_{1}& (t) -V_2(t) \Vert _{\wt{H}_{y_2(t)}^1} \le e^{C\sigma (t-t_0)} \Vert V_{01} -V_{02}  \Vert_{\wt{H}_{y_{2}(t_{0})}^{1}}\\
&+C\int_{t_{0}}^{t} e^{C\sigma (t-s)} \big( \Vert (\Pee\Delta_{1}^{2}) (\cdot)\Vert _{H_{x}^{1}}+ \Vert  \Pi_{c, y_2(s)}^e \big( f_2(s) - f_1(s)\big) \Vert_{H_x^1} \big) ds.
\end{split}
\end{equation}
Using the smoothness and boundedness of $\Pi_{c,y}^e$ and $\mathcal{F}(c,y)$ as well as the fact that $X_{c,y}^\perp$ are finite co-dimensional subspaces of smooth functions, we have
\begin{equation} \label{hgt3}
\begin{split}
&\int_{t_{0}}^{t} \Vert \Pi^{e}_{c,y_{2}} \Delta_1^2 - \Pi^{e}_{c,y_{2}}(JL_{c,y_{1}}-JL_{c,y_{2}})\Pee V_{1}\Vert_{H_x^1}ds \\
\le &C\Vert V_{1}\Vert_{L_{[t_0, t]}^\infty H_x^1} \Vert y_{1}-y_{2}\Vert_{W^{1, 1} ([t_{0},t])}
\end{split}
\end{equation}
Note that 
\begin{equation} \label{E:temp-2.2}
(JL_{c,y_{1}}-JL_{c,y_{2}})\Pee V_{1}= \partial_{x}\left( \tilde Q \Pee V_{1}\right),
\end{equation}
where
$$\tilde Q = \tilde Q (x, y_1, y_2) =k  Q^{k-1}_{c}(\cdot+y_{2})-kQ^{k-1}_{c}(\cdot+y_{1}).$$
Clearly $\frac {\tilde Q}{y_2-y_1}$ and its derivatives in $x$ decay exponentially as $\min \{ |x + y_1|, |x+ y_2|\} \to \infty$. As in the proof of \eqref{ii1}, using \eqref{E:l2li-1}, \eqref{E:sthg-1}, \eqref{E:temp-2.2}, and Lemma \ref{sh1}, 
we obtain 
\begin{equation}\label{hgt2}
\begin{split}
&\int_{t_{0}}^{t_{0}+T} \left\Vert \left(\Pi^{e}_{c,y_{2}}(JL_{c,y_{1}}-JL_{c,y_{2}})\Pee V_{1}\right)(s)\right\Vert _{H^{1}_{x}}ds\\
\le &C(T^{1/2}+ T^{3/2})  \Vert y_{1}-y_{2}\Vert_{L_{[t_{0},t_{0}+T]}^{\infty}}\Vert V_{1}\Vert _{ST^{c}_{[t_{0},t_{0}+T]}}.
\end{split}
\end{equation}
Moreover, Proposition \ref{leve} yields 
$$\Vert V_{1} \Vert _{ST^{c}_{[t_{0},t_{0}+T]}}\le C\big(Ex(T) \Vert V_{01} \Vert_{H^1} + \int_{t_0}^{t_0+T}  Ex(t_0+T-s) \Vert f_1(s) \Vert_{H_x^1}ds  \big).$$
The above estimates 
imply the first two inequalities in the Lemma. The last inequality in the lemma follows similarly by using \eqref{E:eqve3-d}, Proposition \ref{leve}, and the fact that eigenfunctions of $JL_{c,y}$ corresponding to eigenvalues $0$ and $\pm\lambda_c$ are smooth functions, which even allows one to avoid the space-time estimates. 
\end{proof}

\section{Construction of Local Invariant Manifolds of $\mathcal{M}$}\label{construction}
With all the preparation in previous sections, we construct the unique center-stable manifold $\CW^{cs}(\mathcal{M})$ of $\mathcal{M}$, while the center-unstable manifold $\CW^{cu}(\mathcal{M})$can be constructed in a similar manner. The center manifold $\CW^{c}(\mathcal{M})$ of $\mathcal{M}$ is obtained as the intersection of the above center-stable and center-unstable manifolds.

\subsection{Outline of the construction of the center-stable manifold of $\mathcal{M}$.} \label{outline}

We will first fix the traveling speed $c$ and use the global coordinate system \eqref{E:coord-1} to construct the center-stable manifold of the orbit of a single solitary wave 
$$\CM_c=\{Q_c(\cdot+y)\vert y\in \R\}.$$
Eventually we will show that these codim-1 center-stable manifolds $W^{cs}(\CM_c)$ over the directions of $T_{Q_c (\cdot+y)} \CM_c \oplus X_{c, y}^e \oplus X_{c, y}^-$ along $\mathcal{M}_c$, for nearby $c$'s intersect on open subset and thus they can be patched together to form the center-stable manifold of $\CM$. Here in the above, $T_{Q_c (\cdot+y)} \CM_c$ denotes the tangent space of $\CM_c$ at $Q_c(\cdot+y)$. 

In coordinate system \eqref{E:coord-1} $\CW^{cs}(\CM_c)$ is represented as the graph of some mapping $h^{cs}$ 
\begin{equation} \label{E:CU-1} \begin{split}
\CW^{cs}(\CM_c)  = \Phi\big( \big\{ &a^+ = h^{cs} (y, a^{-}, V^e) \mid \\
& (y,  a^-, V^e) \in B^{1} (\delta) \oplus \CX_c^e (\delta)\big\} \big)
\end{split}\end{equation} 
where $\CX_c^e(\delta)$ defined in \eqref{E:CX-delta}. 

Our construction follows the procedure in \cite{JLZ}. Though it has been carried out in details in \cite{JLZ}, for the sake of completeness we briefly describe the procedure here.

To avoid geometric calculations involving bundles, we shall work with $h^{cs}(y, a^-, V)$ defined on $\R\times (-\delta, \delta) \times H^{1} (\delta)$, where $H^{1} (\delta)=\{u|\Vert u \Vert_{H^1}<\delta\}$. However, only the value of $h^{cs}$ on $\R\oplus \CX_c^e(\delta)$ matters. By doing so, the projection operator $ \Pi_{c,y}^e$ will be involved a lot in calculations. The following nonlinear projections will also be used often
\begin{align}
&\wt{\Pi}^e x^{cs} = (y, a^{-}, \Pi_{c, y}^e V ), \quad \text{ where }  x^{cs} = (y, a^{-}, V ). \label{E:wtpi}
\end{align}

Let 
\[
X^{cs} = \R^2\times H^{1}, \quad X^{cs} (\delta) = \{ (y,a^-, V) \in X^{cs} : \Vert V \Vert_{H^1} < \delta\},
\]
and 
\[
X^{cs}_{[t_0,t_1]}=L^{\infty}([t_0,t_1],\R^2)\times ST^{c}_{[t_0,t_1]}, \quad X^{cs}_{[t_0,t_1]}(\delta)=\{ (y,a^-, V) \in X_{[t_0,t_1]}^{cs} : \Vert V \Vert_{ST^{c}_{[t_0,t_1]}} < \delta\}.
\]

As a standard technique in local analysis, we first cut-off the nonlinearities, as well as the off-diagonal linear terms in \eqref{eqyp},  to modify equations \eqref{eqyp}, \eqref{apm} and \eqref{eqve2} into a system defined on $X^{cs} \times \R$. Accordingly, we will work with $h^{cs}(y, a^-, V)$ defined on $X^{cs}(\delta)$. Take a cut-off function 
\begin{equation} \label{E:gamma} 
\gamma\in C_{0}^{\infty}(\mathbb{R}),\;  \text{ s. t. }
\gamma(x)=1, \; \forall \vert x\vert \le 1, \; \; \gamma(x)=0, \;  \forall \vert x\vert \ge 3, \; \; \vert \gamma'\vert_{C^0 (\R)}  \le 1
\end{equation}
and for $\delta >0$, $a^+ \in \R$, and $x^{cs}= (y,a^-, V) \in X^{cs}$, let 
\[
\gamma_\delta (x^{cs}, a^+) = \gamma(\delta^{-1}a^-)\gamma(\delta^{-1}a^+)\gamma(\delta^{-1}\Vert V\Vert_{H_1}). 
\]
and let 
\begin{align*}
\widetilde{G}^{\pm}(x^{cs}, a^+) = &\gamma_{\delta} (x^{cs}, a^+) G^{\pm} (y, a^{+},a^{-},\Pi_{c, y}^eV), \\
\widetilde{G}^T (x^{cs}, a^+) =&  \gamma_{\delta} (x^{cs}, a^+) \big( - \Vert \partial_{x}Q_{c} \Vert_{L^2}^{-2}\langle L_{c,y} J \partial_{x}Q_{c}(\cdot+y), \Pi_{c, y}^e V\rangle \\
&\qquad \qquad \qquad + G^{T}(y,a^{+}, a^{-}, \Pi_{c, y}^eV)\big)\\
\widetilde{G}^e (x^{cs}, a^+) = &\gamma_{\delta} (x^{cs}, a^+) G^{e}(y, a^{+},a^{-},\Pi_{c, y}^eV) 
\end{align*}
where the definitions of $G^{T,\pm,e}$ are given in Section \ref{derivationofeqns}.
The definitions of $\widetilde{G}^{\pm,T,e}$ implies that they are independent of the extra component $(I-\wt \Pi^e)V$ artificially added to avoid the non-flat bundle $\R\oplus \CX^e$. 

Moreover, by the definitions of $G^{\pm,T}$ and the smoothness of the projection operator $\Pi_{c,y}^e$, it holds that for any $m,l\ge 0$, there exists some constant $C$ such that, 
\begin{equation}\label{lipG}
\begin{split}
&\underset{x^{cs},a^+}{\sup}\Vert D_{V}^{m}D_{c, y}^{l}\wt{G}^{T}(x^{cs},a^+)\Vert
\le C\delta^{1-m},\\
&\underset{x^{cs},a^+}{\sup}\Vert D_{a^+,a^-}^{m}D_{c, y}^{l}\wt{G}^{T}(x^{cs},a^+)\Vert
+\underset{x^{cs},a^+}{\sup}\Vert D_{a^+,a^-,V}^{m}D_{c,y}^{l}\wt{G}^{\pm}(x^{cs},a^+)\Vert \le C\delta^{2-m}
\end{split}
\end{equation}
where the above norms are evaluated in the space $\mathcal{L}^{m_l} \left( X^{cs} \times \R, \R\right)$ of $(m+l)-$linear forms on $X^{cs}\times \R $. Denote 
\begin{align*} 
&\widetilde{G}^{cs} (x^{cs}, a^+) = (\widetilde{G}^T, \widetilde{G}^-, \widetilde{G}^e) (x^{cs}, a^+),  \\
&A^{cs} (y, \tilde y) = diag\big(0, -\lambda_{c}, \Pi_{c,y}^{e}JL_{c,y}^e \Pi_{c,y}^{e}+\tilde y \mathcal{F} (c, y) \big).
\end{align*}
We shall consider the following system of $x^{cs}$ and $a^+$, 
\begin{subequations} \label{E:cut}
\begin{equation} \label{E:cut-cs}  
\p_t x^{cs} = A^{cs}\big(y, \widetilde{G}^T (x^{cs}, a^+) \big) x^{cs} + \widetilde{G}^{cs}(x^{cs}, a^+)
\end{equation} 
\begin{equation} \label{E:cut-u}
\p_t a^+ = \lambda_{c} a^+ + \widetilde{G}^+ (x^{cs}, a^+),
\end{equation}
\end{subequations}
which coincides with the system consisting of equations \eqref{eqyp}, \eqref{apm} and \eqref{eqve2} if $|a^{+,-}|\le \delta$, $\Vert V\Vert_{H^1}\le \delta$, and $V \in X_{c, y}^e$. 

The presence of the term 
$\langle L_{c,y} J \partial_{x}Q_{c}(\cdot+y), \Pi_{c, y}^e V\rangle$ in $\widetilde{G}^{T}$ causes  that $\wt G^T$ does not have small Lipschitz constants, which is mostly necessary in constructing local invariant manifolds. 
This technical issue would be handled by introducing the following metric involving a scale constant $A>1$ 
\begin{equation} \label{E:Q-norm} \begin{split}
& \Vert (y, a^-, V) \Vert_{H^1, A} \triangleq |y| + |a^-| + A\Vert V \Vert_{H^1}, \\
& \Vert (y, a^-, V) \Vert_{\wt{H}_{y'}^1, A} \triangleq |y| + |a^-| + A\langle L_{c,y'}\Pi_{c,y'}^e V, \Pi_{c,y'}^eV\rangle^{\frac 1 2} = |y| + |a^-| + A \Vert V \Vert_{\wt H_{y'}^1} \\
& \Vert(y, a^-, V) \Vert_{ST^{c}_{[t_0,t_1]}, A} \triangleq \Vert (y, a^-) \Vert_{L^{\infty}_{[t_0,t_1]}} + A\Vert V \Vert_{ST^{c}_{[t_0,t_1]}}
\end{split}\end{equation} 
where $\Vert \cdot \Vert_{\wt H_{y'}^1}$ was denied in \eqref{tildeh1}. 

After modifying the nonlinearity, we shall construct the local center-stable manifold $\CW^{cs}(\CM_c)$ as the graph $\{ a^+ = h^{cs} (W)\}$ of some $h^{cs}: X^{cs}(\delta) \to \R$. 
In our construction, we fix constants $A, \delta, \mu$ such that 
\begin{equation} \label{E:parameter-1}
\delta<1 ,\quad A>1, \quad \mu<\frac 12, 
\end{equation}
with additional assumptions which will be given later. Define
\begin{equation} \label{E:Gamma-mu} 
\Gamma_{\mu, \delta} = \{h : X^{cs} (\delta)  \to \R  \mid  h(y, 0,0) =0, \;  \Vert h\Vert_{C^0} \le \delta,  Lip(h)_{\Vert \cdot \Vert_{H^1, A}} \le \mu.
\}.
\end{equation}
Here $h(y,0,0) =0$ is required since $\CW^{cs}(\CM_c)$ should contain $\CM_c$. 
It is clear that $\Gamma_{\mu, \delta}$ equipped with $\Vert \cdot\Vert_{C^0}$ is a complete metric space. 

We will perform a type of Lyapunov-Perron method to construct center-stable manifolds. That is,  for any $h \in \Gamma_{\mu, \delta}$ and $\bar x^{cs}\in X^{cs}(\delta)$, let $x^{cs}(t) = (y, a^-, V^e)(t) \in X^{cs}$ be the solution to 
\begin{equation}\label{lp}
\p_t x^{cs} = A^{cs} \big(y, \widetilde{G}^T(x^{cs}, h(x^{cs})\big) \big)x^{cs} + \widetilde{G}^{cs} \big(x^{cs}, h(x^{cs}) \big), \quad x^{cs}(0)= \bar x^{cs}. 
\end{equation}
Then we define $\widetilde{h}(\bar{x}^{cs})$ as 
\begin{equation}\label{newh}
\widetilde{h} (\bar{x}^{cs})
= \bar a^+ =-\int_{0}^{\infty}e^{-\lambda_{c} s}\widetilde{G}^{+} \big(x^{cs}(s), h(x^{cs}(s))\big) ds.
\end{equation}

\begin{remark}\label{R:C-0}
Even though $h$ is defined only on $X^{cs}(\delta)$, due to the cut-off function $\gamma_{\delta}$, 
for any $h \in \Gamma_{\mu, \delta}$, $\alpha \in \{T,  \pm, e\}$, it holds $\widetilde{G}^{\alpha} \big(x^{cs}, h(x^{cs})\big) =0$ whenever $x^{cs} \in X^{cs}\backslash X^{cs}(\delta)$. Thus,  the right side of \eqref{lp} is well-defined for all $x^{cs}\in X^{cs}$. 
\end{remark}

Denote the transformation $h \to \wt h$ as 
\[
\mathcal{T} (h) = \wt h.
\]

The aim is to show that, under suitable assumptions on $A$, $\delta$ and $\mu$ , $\wt h \in \Gamma_{\mu, \delta}$ is well-defined and $\CT$ is a contraction on $\Gamma_{\mu, \delta}$. The graph of the unique fixed point, restricted to the set 
\[
B^{1} (\delta) \oplus \CX^e (\delta) 
\]
would be the desired center-stable manifold $\CW^{cs}(\CM_c)$. 

The framework described above allows us to work in a flat space $X^{cs}$ instead of non-flat bundle $\R\oplus \CX^e$, which will bring us convenience in the proof of the smoothness of local invariant manifolds. In fact, those extensions and modifications of \eqref{eqyp}, \eqref{apm} and \eqref{eqve2} to cooperate with our framework does not change the local invariant manifolds. More precisely, we have the following lemma. 

\begin{lemma} \label{L:reduce} 
The following statements hold. 
\begin{enumerate} 
\item Suppose $$x^{cs}(t)=(y(t),a^{-}(t), V(t))$$ satisfies \eqref{E:cut-cs} on $[t_1, t_2]$ for some $a^+ \in C^0([t_1, t_2], \R)$ and $ \wt{\Pi}^ex^{cs}(t_0)=x^{cs}(t_0)$ for some $t_0 \in [t_1, t_2]$, then $\wt{\Pi}^e x^{cs}(t) =x^{cs}(t)$ for all $t \in [t_1, t_2]$. 
\item Assume $h_j \in \Gamma_{\mu, \delta}$, $j=1,2$, satisfy $h_1(x^{cs})=h_2(x^{cs})$ for all $x^{cs} \in \R\oplus \CX_c^{e}(\delta)$. Then $\tilde h_j$, $j=1,2$, defined in \eqref{newh} are also identical in $\R\oplus \CX_c^{e}(\delta)$. 
\end{enumerate}
\end{lemma} 

\begin{proof}
Since $\Pi_{c, y}^e \wt G^e=0$, this lemma is just an easy application of Lemma \ref{sh1}. 
\end{proof}

\subsection{Apriori estimates}
In this subsection, we utilize the smoothing space-time estimates established in Section \ref{LinearAnalysis} to obtain apriori estimates. The strategy is to derive small time period estimates with small exponential growth, then by iteration we obtain global in time estimates with the same exponential growth. The Hamiltonian structure plays a crucial role in our iteration step. In particular, the positivity of the bilinear form $\langle L_{c,y}\cdot, \cdot\rangle$ in $X_{c,y}^e$ guarantees the coefficient in front of the exponential term is $1$, so iteration will not generate large exponential growth. 

We start the subsection with several  estimates that will be used frequently throughout this paper. 

\begin{lemma}\label{bilinear}
If $u,v\in ST^c_{[t_0,t_0+T]}$, the following bilinear estimate holds $$\Vert \p_x(uv)\Vert_{L^1_{[t_0,t_0+T]}H_x^1}\le C(T^{\frac 12} +T)\Vert u \Vert_{ST^{c}_{[t_0,t_0+T]}}\Vert v \Vert_{ST^{c}_{[t_0,t_0+T]}}. $$
\end{lemma} 

\begin{proof}
Let $\tilde u(t, x) = u(t, x-ct)$ and apply the same notation convention to $v$. Since $\Vert \p_x(uv) \Vert_{L_{[t_0, t_0+T]}^1 L_x^2} = \Vert \p_x(\tilde u \tilde v) \Vert_{L_{[t_0, t_0+T]}^1 L_x^2}$, we may estimate the latter in terms of the $ST'_{[t_0,t_0+T]}$ norm of $\tilde u$ and $\tilde v$. Firstly, since $H^1 \subset L^{\infty}$, one can first estimate 
$$\Vert (\p_x\tilde  u)\tilde v+(\p_x \tilde v)\tilde u\Vert_{L^1_{[t_0,t_0+T]}L^2_x}\le CT\Vert \tilde u\Vert_{L^\infty_{[t_0,t_0+T]}H_x^1}\Vert \tilde v\Vert_{L^\infty_{[t_0,t_0+T]}H_x^1}.$$
Moreover, by straightforward calculation, one has 
\begin{equation*}
\begin{split}
\Vert \partial_{x}\tilde u \p_x \tilde v\Vert _{L^1_{[t_0,t_0+T]}L_{x}^{2}}
\le & \Vert \partial_{x}\tilde u\Vert _{L^{4}_{[t_{0},t_{0}+T]}L_{x}^{\infty}}\Vert \partial_{x}\tilde v \Vert _{L^{4/3}_{[t_{0},t_{0}+T]}L_{x}^{2}}\\
\le &T^{3/4}\Vert \partial_{x}\tilde u\Vert _{L^{4}_{[t_{0},t_{0}+T]}L_{x}^{\infty}}\Vert \partial_{x}\tilde v\Vert _{L^{\infty}_{[t_{0},t_{0}+T]}L_{x}^{2}},\\
\end{split}
\end{equation*}
and 
\begin{equation*}
\begin{split}
\Vert \tilde u \p_{xx} \tilde v\Vert _{L^1_{[t_0,t_0+T]}L_{x}^{2}}
\le & T^{1/2}\Vert \tilde u\partial_{xx}\tilde v\Vert _{L_{x}^{2}L^{2}_{[t_{0},t_{0}+T]}}\\
\le &T^{1/2}\Vert \tilde u\Vert _{L^{2}_{x}L^{\infty}_{[t_{0},t_{0}+T]}}\Vert \partial_{xx}\tilde v\Vert _{L^{\infty}_{x}L^{2}_{[t_{0},t_{0}+T]}}.
\end{split}
\end{equation*}
The term $\Vert \tilde v \p_{xx} \tilde u\Vert _{L^1_{[t_0,t_0+T]}L_{x}^{2}}$ can be estimated similarly, which completes the proof. 
\end{proof}

Next technical lemma will be used frequently to estimate the difference between two solutions to \eqref{lp}.

\begin{lemma}\label{e12}
Let $V \in L^{\infty}_{[t_{0},t_{0}+T]}L^{2}_{x}$ and $y(t) \in C^{1}([t_0,t_0+T],\R)$. For any $m\ge 0$ it holds 
$$
\Vert (\p_{y}^{m}\Pi_{c,y}^{e})|_{y=y(t)} V\Vert_{
ST^{c}_{[t_{0},t_{0}+T]}
}\le C (1 + T^{\frac 12})\big(1+(c+\Vert \p_t y(t)\Vert_{C^0})T\big)
\Vert V\Vert_{L^{\infty}_{[t_{0},t_{0}+T]}L^{2}_{x}}.
$$
\end{lemma}

\begin{proof}
Since $\p_{y}^{m}\Pi_{c,y}^{e}=-\p_{y}^{m}\Pi_{c,y}^{\perp}$, it is equivalent to estimate $\Vert \p_{y}^{m}\Pi_{c,y}^{\perp}V\Vert_{
ST^{c}_{[t_{0},t_{0}+T]}
}$. On the one hand, by Lemma \ref{l2li}, we have
\[\begin{split}
& \Vert \p_x^m V^{\pm}_{c}(\cdot+y(t))\Vert_{ST^{c}_{[t_{0},t_{0}+T]}} + \Vert \partial_x^m Q_{c}(\cdot+y(t))\Vert_{ST^{c}_{[t_{0},t_{0}+T]}} \\
\le & C(1 + T^{\frac 12})\big(1+(c+\Vert \p_t y(t)\Vert_{C^0})T\big),
\end{split}\]

On the other hand, using the high regularity of $V^{\pm}_c$ and $Q_c$, one has 
$$\vert\langle  \p_x^m V^{\pm}_{c}(\cdot+y(t), V\rangle \vert + \vert\langle  \partial_x^m Q_{c}(\cdot+y(t)), V\rangle \vert \le C\|V\|_{L^2_x}. $$
By the explicit expressions of $\Pi_{c,y}^{\pm,T}$ given in Lemma \ref{decompvb}, the desired estimate follows right away. 
\end{proof}

An immediate consequence of the above two lemmas along with Lemma \ref{l2li} is the following. 

\begin{lemma}\label{lipGe}
For any $x^{cs}\in X^{cs}_{[t_0,t_0+T]}(a)$ and $m_j \ge 0$, $j=1, \ldots, 5$, with $m_4+m_5>0$, it holds that 
$$\Vert D_V^{m_1} \p_{a^\pm}^{m_2} \p_{y}^{m_3}\wt{G}^{e}(x^{cs},a^+)\Vert_{\mathcal{L}^{m_1} (ST_{[t_0,t_0+T]}^c, L^1_{[t_0,t_0+t]}H_x^1)}\le CT^{\frac 1 2} (1+T^2)(a+\delta)^{2-m_1-m_2}$$
\[
\Vert D_V^{m_4} \p_{a^\pm, y}^{m_5} \big( A^{cs} (y, \wt G^T(x^{cs}, a^+)) \big) \Vert_{\mathcal{L}^{m_4}\big( L_{[t_0, t_0+T]}^\infty H_x^1, \CL( X^{cs}_{[t_0,t_0+T]}, L^1_{[t_0,t_0+t]}H_x^1)\big)}\le C T^{\frac 12} (1 + T).
\]
\end{lemma}

Here $\mathcal{L}^l \big(Z_1, Z_2\big)$ denotes the space of $l-$linear transformations from space $Z_1$ to space $Z_2$.  In the above differentiations, $\p_{a^\pm}^{m_2} \p_{y}^{m_3}\wt{G}^{e}$ and $\p_{a^\pm, y}^{m_5} A^{cs}$ are point-wise partial derivatives and the multi-linear operators resulted in the differentiations are of $V$ only.

\begin{proof}
We first consider $\wt G^e$. For convenience, we let 
$$\widetilde{G}(x^{cs},a^+)=\gamma_{\delta}(x^{cs},a^+)G(y,a^{+},a^{-},V^{e}),$$
and
$$\widetilde{R}^{e}(x^{cs},a^+)=\widetilde{G}(x^{cs},a^+)-\widetilde{G}^e(x^{cs},a^+) = \gamma_{\delta}(x^{cs},a^+) \big( G(x^{cs},a^+)- G^e(x^{cs},a^+) \big).$$
Recalling the definitions of $G$ and $G^e$ in\eqref{E:G-1} and \eqref{eqve}, respectively, the difference between $G^e$ and $G$ consists of terms of high spatial regularity smoothly depending on $c, y, a^\pm \in \R$ and $V \in L^2$, it is straightforward to verify 
$$\Vert  D_V^{m_1} \p_{a^\pm}^{m_2} \p_{y}^{m_3} \widetilde{R}^{e}(x^{cs},a^+)\Vert _{\mathcal{L}^{m_1}  (H_x^1 , H_x^1)}\le C \delta^{2-m_1-m_2},$$ 
which implies 
$$\Vert  D_V^{m_1} \p_{a^\pm}^{m_2} \p_{y}^{m_3} \widetilde{R}^{e}(x^{cs},a^+)\Vert _{\mathcal{L}^{m_1}   \big(ST_{[t_0,t_0+T]}^c, L^1_{[t_0,t_0+t]}H_x^1\big)} \le CT \delta^{2-m_1-m_2}. $$
Since $G$ is a polynomial of $a^\pm$ and $V$, using the fact that $H^1(\R)\subset L^\infty (\R)$, Lemma \ref{l2li}, and Lemma \ref{bilinear}, one has 
\begin{equation*}
\Vert  D_V^{m_1} \p_{a^\pm}^{m_2} \p_{y}^{m_3} \wt{G}(x^{cs},a^+)\Vert _{\mathcal{L}^{m_1} (S_{[t_0,t_0+T]}^c, L^1_{[t_0,t_0+t]}H_x^1)} \le CT^{\frac 12}(1+T^2) (\delta+a)^{2-m_1-m_2}.
\end{equation*}

To estimate $D_V^{m_4} \p_{a^\pm, y}^{m_5} A^{cs} \big(y, \wt G^T(x^{cs}, a^+)\big)$ for $m_4+m_5>0$, we first consider 
\[
\p_y^l (JL_{c, y}) V   =k  \p_x \big( \p_x^l (Q_c^{k-1}) (\cdot + y) V\big).
\]
Much as in the proof of \eqref{ii1}, we obtain, 
\[
\Vert \p_y^l (JL_{c, y}) (\bar y_1, \ldots, \bar y_l) V \Vert_{L^1_{[t_0,t_0+T]}H_x^1} \le C T^{\frac 12} (1 + T)  \Vert V \Vert_{ST_{[t_0, t_0+T]}^c}. 
\]
Due to high regularity of the eigenfunctions of $JL_{c, y}$, the smoothness of $\Pi_{c, y}^\perp$ and $\CF$ with respect to $y$, and the smoothness of  $\wt G^T$ with respect to $x^{cs} \in X^{cs}$ and $a^+ \in \R$, The inequality on $D_V^{m_4} \p_{a^\pm, y}^{m_5}  A^{cs}$ follows immediately and this completes the proof.
\end{proof}

In the rest of this subsection, we shall solve and estimate solutions to \eqref{E:cut-cs} with a given $a^+(t)$. One first observes that the $V$ equation has to be solved along a path $y(t)$ and the multiplier in front of $\CF$ has to be its $\p_t y$ in order maintain the commutativity (obtained in Lemma \ref{sh1}) between its homogeneous evolution operator $S(t, t_0)$ and $\Pi_{c, y}^e$. Therefore we split the iteration procedure of \eqref{E:cut-cs} into 
\begin{subequations} \label{E:cut-split}
\begin{equation} \label{E:cut-cs-T}
\p_t y = \wt G^T ( \wt y, \wt a^-, \wt V, \wt a^+)
\end{equation}
\begin{equation}  \label{E:cut-cs--}
\p_t a^- = -\lambda_{c} a^- + \wt G^- (\wt y, \wt a^-, \wt V, \wt a^+)
\end{equation}
\begin{equation}  \label{E:cut-cs-V}
\p_t V= \big(\Pi_{c,y}^{e}JL_{c,y}^e \Pi_{c,y}^{e}+\p_t y \mathcal{F} (c, y) \big) V + \wt G^e ( \wt y, \wt a^-, \wt V, \wt a^+)
\end{equation}
\end{subequations}
where 
\begin{equation} \label{E:temp-2.4}
\wt x^{cs} = ( \wt y, \wt a^-, \wt V) \in X_{[t_0, t_0+T]}^{cs}(a), \quad a\in (0, 1),  \qquad \wt a^+ \in L_{[t_0, t_0+T]}^\infty
\end{equation}
are given. In particular, ones first solves the ODEs \eqref{E:cut-cs-T} and \eqref{E:cut-cs--} for $y(t)$ and $\wt a^-(t)$ and then substitutes the solution $y(t)$ into the homogeneous part of equation \eqref{E:cut-cs-V} and solves for $V(t)$. 

\begin{lemma}\label{aprior3}
Let $\wt x_i^{cs}(t)$ and $\wt a_i^+(t)$ satisfy \eqref{E:temp-2.4}, $i=1,2$, and $x_i^{cs}(t)=\big(y_i(t),a_i^-(t), V_i(t)\big)$ be the solutions to \eqref{E:cut-split} for $t\in [t_0, t_0+T]$, then there exist a constant $C$ not depending on $t_0, T, x^{cs} (t_0)$, and $\wt x^{cs}$ such that, if initial value \[x_i^{cs}(t_0)=x_{i0}^{cs}=\big(y_{i0},a_{i0}^-, V_{i0}\big)\in X^{cs}(C\delta),\] then we have $\Vert \p_t y_i \Vert_{L^\infty} \le C \delta$ and 
\[
|y_i(t_0+T)| \le |y_{i0}| + C \delta T, \quad |a_i^-(t_0+T)| \le e^{-\lambda_c T} |a_{i0}^-| + C\delta
\]
\[
\Vert V_i\Vert_{ST^{c}_{[t_0,t_0+T]}} \le  C(1+T^6)e^{C\delta T}\big(\Vert  V_{i0} \Vert_{H^1}+ T^{\frac 12} (a^2+ \delta^2) \big),
\]
\begin{align*}
|(y_2 -y_1)(t_0+T)| 
\le& |y_{20} - y_{10}| + C T (\delta + A^{-1}) (\Vert \wt x_2^{cs} - \wt x_1^{cs}\Vert_{ST_{[t_0, t_0+T]}^c, A} + \Vert \wt a_2^+ - \wt a_1^+\Vert_{L_t^\infty})
\end{align*}
\begin{align*}
|(a_2^- - a_1^-) (t_0+T)| \le & e^{-\lambda_c (t-t_0)} |a_{20}^- - a_{10}^-| + C \delta \int_{t_0}^{t_0+T} e^{-\lambda_c (t-\tau)} (|\wt y_2 - \wt y_1|\\
&\qquad \qquad + |\wt a_2^- -\wt a_1^-| + |\wt V_2 - \wt V_1|_{H^1} + |\wt a_2^+ -\wt a_1^+|) (\tau) d\tau
\end{align*}

\begin{align*}
\Vert V_2 - & V_1  \Vert_{ST^{c}_{[t_0,t_0+T]}} 
\le  C(1+T^9)e^{C\delta T}\Big(\Vert  V_{20} - V_{10} \Vert_{H^1}+ T^{\frac 12} \big( a+ \delta \\
&+ T^{\frac 12} (a^2+ \delta^2)\big)  
(|y_{20} - y_{10}|+ \Vert \wt x_2^{cs} - \wt x_1^{cs}\Vert_{ST_{[t_0, t_0+T]}, A} + \Vert \wt a_2^+ - \wt a_1^+\Vert_{L_t^\infty}) \Big),
\end{align*}
\begin{equation*}
\begin{split}
\Vert  V_2(t_0&+T)  -  V_1 (t_0+T)  \Vert_{\wt{H}_{y_2(t_0+T)}^1}\le  e^{C\delta T}\big(\Vert V_{20}- V_{10} \Vert_{\wt{H}_{y_2(t_0)}^1} + C T^{\frac 12} (1+T^9) \\
&\times \big( a+ \delta + T^{\frac 12} (a^2+ \delta^2)\big) (|y_{20} - y_{10}|+ \Vert \wt x_2^{cs} - \wt x_1^{cs}\Vert_{ST_{[t_0, t_0+T]}, A} + \Vert \wt a_2^+ - \wt a_1^+\Vert_{L_t^\infty}) \Big),
\end{split}
\end{equation*}
and for any $l>0$
\begin{equation*}\begin{split}
\Vert \big( (I-\Pi_{c, y_2}^e) ( V_2  - V_1) \big)  (t_0+T) &\Vert_{H^l} \le  e^{C\delta T}\big( \Vert (I-\Pi_{c, y_2(t_0)}^e)(V_{02}  - V_{01})  \Vert_{H^1} \\
& + C T^{\frac 12} (1+T^4) \big( a+ \delta + T^{\frac 12} (a^2+ \delta^2)\big) (|y_{20} - y_{10}|\\
&+ \Vert \wt x_2^{cs} - \wt x_1^{cs}\Vert_{ST_{[t_0, t_0+T]}, A} + \Vert \wt a_2^+ - \wt a_1^+\Vert_{L_t^\infty}) \Big).
\end{split}\end{equation*}
\end{lemma}

\begin{proof}
From \eqref{E:cut-split}, \eqref{lipG}, Proposition \ref{leve}, and Lemma \ref{lipGe}, it is straightforward to obtain the estimates on $\wt x^{cs}_i$ and compute 
\begin{align*}
|(y_2 -y_1)(t)| \le |y_{20} - y_{10}| + C \int_{t_0}^{t} (&\delta |\wt y_2 - \wt y_1| + \delta |\wt a_2^- -\wt a_1^-| \\
&+ |\wt V_2 - \wt V_1|_{H^1} + \delta |\wt a_2^+ -\wt a_1^+|) (\tau) d\tau. 
\end{align*}
Therefore the estimate on $y_2 - y_1$ follows immediately from the definition of $\Vert \cdot \Vert_{ST_{[t_0, t_0+T]}^c, A}$. The inequality on $a_2^- - a_1^-$ is also derived from \eqref{E:cut-split} and \eqref{lipG}. To estimate the difference in $V_i(t)$, we first obtain from  Lemma \ref{lipGe} 
\[
|\p_t y_i|_{L_t^\infty} \le C\delta, \qquad \Vert \wt G^e (\wt x_i^{cs}, \wt a_i^+) \Vert_{L_t^1 H_x^1} \le CT^{\frac 12} (1+T^2) (a^2+\delta^2), 
\]
for $i=1,2$ where the the norms in $t$ are taken on $[t_0, t_0+T]$ throughout the lemma. Using Lemma \ref{diff-hom}, we have 
\begin{align*} 
\Vert V_2& - V_1 \Vert_{ST^{c}_{[t_0,t_0+T]}}  \le  C(1+T^7)e^{C\delta T}\Big(\Vert V_{02} - V_{01} \Vert_{H^1} \\
&+ \Vert \wt G^e (\wt x_2^{cs}, \wt a_2^+) - \wt G^e (\wt x_1^{cs}, \wt a_1^+) \Vert_{L_t^1 H_x^1} + T^{\frac 12} \big(\delta + T^{\frac 12} (a^2+ \delta^2)\big) \Vert y_2 - y_1\Vert_{C^{0,1}}\Big).
\end{align*}

Again from Lemma \ref{lipGe}, \eqref{E:cut-split},  \eqref{lipG} and the inequality on $y_2 - y_1$, we have
\begin{align*}
& \Vert \wt G^e (\wt x_2^{cs}, \wt a_2^+) - \wt G^e (\wt x_1^{cs}, \wt a_1^+) \Vert_{L_t^1 H_x^1} \le CT^{\frac 12} (1+T^2) (a+\delta) (\Vert \wt x_2^{cs} - \wt x_1^{cs}\Vert_{X_{[t_0, t_0+T]}^{cs}} + \Vert \wt a_2^+ - \wt a_1^+\Vert_{L_t^\infty})\\
 &\Vert y_2 - y_1\Vert_{C^{0,1}} \le  |y_{20} - y_{10}| + C (1+T) (\delta + A^{-1}) (\Vert \wt x_2^{cs} - \wt x_1^{cs}\Vert_{ST_{[t_0, t_0+T]}^c, A} + \Vert \wt a_2^+ - \wt a_1^+\Vert_{L_t^\infty}).
\end{align*}
The last inequality follows similarly and the proof is complete. 
\end{proof}

\begin{remark} \label{R:apriori3}
Clearly \eqref{E:cut-split} is well-posed and $x^{cs} =(y_0, 0, 0)$ if $\wt x^{cs}\equiv \wt x^{cs}(t_0) = (y_0, 0,0)$ and $\wt a^+ =0$. 
\end{remark}

With the above lemma, we are ready to prove the well-posedness of \eqref{E:cut-cs}. 

\begin{lemma} \label{L:WP1}
Given any $C_0>1$, there exists $C>1$ such that if $A$ and $\delta$ satisfy \eqref{E:parameter-1} and 
\begin{equation} \label{E:parameter-2}
C\eta <1, \quad \text{ where } \quad \eta= A\delta + A^{-1}
\end{equation} 
then 
\begin{enumerate}
\item For any $a^+ \in L_{loc}^\infty$ and $x_0^{cs} = (y_0, a_0^-, V_0)  \in X^{cs}(C_0\delta)$, there exists a unique solution $x^{cs} = (y, a^-, V) \in C^0 \big([0, \infty), X^{cs}(C\delta)\big)$ of \eqref{E:cut-cs} such that $x^{cs} (0) = x_0^{cs}$ and $x^{cs}  \in X_{[t_0, t_0+T]}^{cs} \big(C(1+T)\delta\big)$ for any $t_0, T\ge 0$. 
\item Let $x_i^{cs} = (y_i, a_i^-, V_i)$ be the solutions of \eqref{E:cut-cs} with initial values $x_{i0}^{cs} = (y_{i0}, a_{i0}^-, V_{i0})  \in X^{cs}(C\delta)$ corresponding to $a_i^+ \in L_{loc}^\infty$, $i=1,2$. Suppose 
\begin{equation}\label{a+12}
|a^{+}_{1}-a^{+}_{2}|\le \kappa_0 + \kappa_1 \Vert x_1^{cs}-x_2^{cs}\Vert_{H^{1}, A},
\end{equation}
for $\kappa_0>0$ and $\kappa_1 \in [0, 1]$, then  
\begin{align*}
&\Vert x^{cs}_1(t)-  x^{cs}_2(t)\Vert_{H^1,A}\le Ce^{ C\eta t}\left( \Vert x^{cs}_{10}- x^{cs}_{20}\Vert_{H^1,A}+ 
 \min \{1, (1+t)\eta)\}\kappa_0  \right), \\
& \Vert x_2^{cs} - x_1^{cs} \Vert_{ST^{c}_{[t_0, t_0+T]}, A} \le  C (1+T) e^{C \eta T}  \big(  \Vert x_{2}^{cs} (t_0) - x_{1}^{cs} (t_0) \Vert_{H^1, A}  +  \kappa_0\big).
\end{align*}
\end{enumerate}
\end{lemma}

\begin{proof}
In the proof of this lemma, we will use $C'$ to denote the generic upper bounds appearing in previous estimates and $C$ the newer (and greater) bound emerging in the proof of this lemma. For any $\wt x^{cs} \in X_{[0, 1]}^{cs} (C C_0 \delta )$, let $x^{cs}= (y, a^-, V)$ be the solution to \eqref{E:cut-split} with $\wt a^+=a^+$ and the  initial value $x_0^{cs}$. From Lemma \ref{aprior3} and \eqref{E:parameter-1}, we have 
\[
\Vert V \Vert_{ST^{c}_{[0, 1]}}
\le C' \big(C_0 \delta+ C^2 C_0^2 \delta^2) \big) \le CC_0 \delta.
\]

Therefore $x^{cs} \in X_{[0, 1]}^{cs} (CC_0 \delta)$. Moreover, the mapping $\wt x^{cs} \to x^{cs}$ has the Lipschitz constant $C' \eta<1$ in the $\Vert \cdot \Vert_{ST_{[0, 1]}^c, A}$ norm
due to \eqref{E:parameter-2}. The Contraction Mapping principle implies the local well-posedness of \eqref{E:cut-cs} for $t \in [0,1]$. 

Note that when $\Vert V\Vert _{H^1}>3\delta$,  we have $\partial_{t}y=0$ and $\widetilde{G}^{e}=0$ in \eqref{E:cut-cs}. Consequently, by \eqref{vevp1} and \eqref{vevp2}, 
$$\p_t\langle L_{c,y}V^{e},V^{e}\rangle=0,\qquad\qquad \p_t V^\perp=0,$$ 
where $V^e=\Pi_{c,y}^e V$ and $V^\perp=\Pi_{c,y}^\perp V$, which along with the positivity of $L_{c, y}$ on $X_{c, y}^e$ implies 
\begin{equation}\label{bov}
 \Vert V\Vert _{H^{1}}  \le C\delta
\end{equation}
for any $t\in [0,1]$. Therefore a standard continuation argument yields the global in time well-posedness of \eqref{E:cut-cs} with $x^{cs} \in X_{[t_0,t_0+T]}^{cs} \big(C(1+T)\delta\big)$ for any $t_0, T\ge 0$. 

To prove the second part of the lemma, we first notice that Lemma \ref{aprior3} implies that  
\begin{align*}
\Vert x_2^{cs} - x_1^{cs} \Vert_{ST^{c}_{[t_0, t_0+1]}, A} \le  C'  \big(&  \Vert x_{2}^{cs} (t_0) - x_{1}^{cs} (t_0) \Vert_{H^1, A} \\
& + ( CC_0 A\delta +A^{-1} ) \big( \Vert x_2^{cs} - x_1^{cs} \Vert_{ST^{c}_{[t_0, t_0+1]}, A} +\kappa_0\big) \big)
\end{align*}
where $x_i^{cs} \in X_{[t_0,t_0+1]}^{cs} (CC_0\delta)$ is used. From \eqref{E:parameter-1}, we obtain 
\begin{equation} \label{E:temp-2.8}
\Vert x_2^{cs} - x_1^{cs} \Vert_{ST^{c}_{[t_0, t_0+1]}, A} \le  C \big(  \Vert x_{2}^{cs} (t_0) - x_{1}^{cs} (t_0) \Vert_{H^1, A}  + \eta \kappa_0\big).
\end{equation}

Let 
\[
l(t) = \Big(|y_2 - y_1| + |a_2^- - a_1^-| + A\Vert V_2  - V_1  \Vert_{\wt H_{y_2}^1} + A\Vert \big( (I-\Pi_{c, y_2}^e) ( V_2  - V_1) \big) \Vert_{H^1} \Big)\Big|_t, 
\]
which satisfies 
\[
(1/C') l(t) \le \Vert x_2^{cs} (t) - x_1^{cs} (t) \Vert_{H^1, A} \le C' l(t). 
\]
Substituting \eqref{E:temp-2.8} into Lemma \ref{aprior3} yields, for $t\in [0, 1]$, 
\begin{equation} \label{E:temp-3} 
l(t_0+t) \le e^{C'\delta t}  \big( l(t_0) + C\eta t^{\frac 12}  (l(t_0) + \kappa_0)\big). 
\end{equation}
In particular it implies 
\[
l(n+1) \le e^{C\eta}  \big( l(n) + C \eta \kappa_0\big).
\]
From a simple induction argument we obtain
\[
l(n) \le e^{C\eta n}  \big( l(0) + C 
\min\{1, n \eta\}\kappa_0\big)
\]
which along with \eqref{E:temp-3} implies, for $t\ge 0$, 
\[
l(t) \le 2 e^{C\eta t}  \big( l(0) + C \min\{1,  (1+t) \eta\} 
 \kappa_0  \big). 
\]
Therefore we obtain the desired estimate on $\Vert x_1^{cs} - x_2^{cs}\Vert_{H^1, A}$. Finally, substituting this into \eqref{E:temp-2.8}, we obtain 
\[
\Vert x_1^{cs} - x_2^{cs}\Vert_{ST_{[n, n+1]}^c, A} \le C e^{C\eta n} ( \Vert  x_{10}^{cs} -  x_{20}^{cs} \Vert_{H^1, A} + \kappa_0),
\]
the summation of which implies the estimates $\Vert x_1^{cs} - x_2^{cs}\Vert_{ST_{[t_1, t_2]}^c, A}$. 
\end{proof}

\subsection{Construction of local center-stable manifolds}
In this section, we follow the procedure described in Section \ref{outline} to construct center-stable manifolds of $\CM_c$. The goal is to show the transformation $h\rightarrow \wt{h}$ is a contraction on $\gamma_{\mu,\delta}$, where  $\wt{h}$  and $\gamma_{\mu,\delta}$ are defined in \eqref{newh} and \eqref{E:Gamma-mu}, respectively. We first give the global well-posedness of \eqref{lp} in the below.

\begin{lemma}\label{WP}
There exists $C>1$ such that if $A$, $\mu$, and $\delta$ satisfy \eqref{E:parameter-1} and \eqref{E:parameter-2}, then, for any  $x_0^{cs} = (y_0, a_0^-, V_0)  \in X^{cs}(\delta)$, there exists a unique solution $x^{cs} = (y, a^-, V) \in C^0 \big([0, \infty), X^{cs}(C\delta)\big)$ of \eqref{lp} such that $x^{cs}(0) = x_0^{cs}$ and $x^{cs}  \in X_{[t_0, t_0+T]}^{cs}$ for any $t_0, T\ge 0$. 
\end{lemma}

\begin{remark}
The global well-posedness of \eqref{lp} can be proved by the same arguments as in the proof Lemma \ref{L:WP1} and we omit it. The estimate in part (2) of Lemma \ref{L:WP1} obviously holds for solutions to \eqref{lp}.  
\end{remark}

Since \eqref{lp} is global well-posed, the definition  \eqref{newh} of $\wt{h}$ is valid. Next, we show that the map $\CT(h)=\wt{h}$ is a contraction on $\gamma_{\mu,\delta}$. 

\begin{proposition}\label{contraction}
There exists $C>1$ such that if $\delta$, $\mu$, and $A$ satisfy \eqref{E:parameter-1}, \eqref{E:parameter-2}, and 
\begin{equation}\label{parameter-set-2}
C (\lambda_c - C\eta )^{-1}\delta\le  \mu,
\end{equation}
then $\mathcal{T}$ is a contraction mapping on $\Gamma_{\mu,\delta}$.
\end{proposition}

\begin{proof}
By \eqref{lipG} and the definition  \eqref{newh} of $\wt{h}$,  $\Vert \widetilde{h}\Vert _{C^0}\le C  \lambda_{c}^{-1}\delta^{2}\le \delta$. Since $G^{-,e}(y,0,0,0)=h(c,y,0,0)=0$,  when the initial data $\bar{x}^{cs}=(\bar y, 0,0)$, the solution to \eqref{lp} is of the form $(c,y,0,0)$, which implies $\widetilde{h}(y,0,0)=0$. 

For $h_i \in \Gamma_{\mu, \delta}$, let $x^{cs}_i$ be the solutions to \eqref{lp} with $h =h_i$ and with the initial data $\bar x_i^{cs} \in X^{cs} (\delta)$, $i=1,2$, respectively. Since 
\[
\Vert h_1(x_1^{cs})-h_2(x_2^{cs})\Vert \le \mu \Vert x_1^{cs}-x_2^{cs} \Vert_{H^1, A} + \Vert h_1 - h_2 \Vert_{C^0},
\]
applying Lemma \ref{L:WP1}, we have 
\begin{equation}\label{diff-x12cs}
\Vert x_1^{cs}(t)-x_2^{cs}(t)\Vert_{H^1,A}\le Ce^{C\eta t}(\Vert \bar x_1^{cs}-\bar x_2^{cs}\Vert_{H^1,A}+
\Vert h_1-h_2\Vert _{C^0}).
\end{equation}

In \eqref{diff-x12cs}, letting $h_1=h_2=h$, we have 
$$\Vert x_1^{cs}(t)-x_2^{cs}(t)\Vert_{H^1,A}\le Ce^{\eta t}\Vert \bar x_1^{cs}-\bar x_2^{cs}\Vert_{H^1,A}.$$
It follows from \eqref{newh} and \eqref{lipGe} that 
\begin{equation}
\left\vert \widetilde{h}(\bar{x}^{cs}_1)-\widetilde{h}(\bar{x}^{cs}_2)\right\vert \le C\big(\lambda_c - C\eta \big)^{-1} \delta\Vert \bar x_1^{cs}-\bar x_2^{cs}\Vert_{H^1,A},
\end{equation}
which implies that $Lip(\wt{h})_{\Vert \cdot \Vert_{H^1, A}}\le \mu$ due to \eqref{parameter-set-2}.

Applying \eqref{lipGe} and \eqref{diff-x12cs} with $\bar x_1^{cs}=\bar x_2^{cs}$ in \eqref{newh}, we have 
\begin{equation}
\Vert \widetilde{h}_{1}-\widetilde{h}_{2}\Vert _{C^0}\le C \big(\lambda_c - C\eta\big)^{-1} \delta \Vert h_{1}-h_{2}\Vert _{C^0}.
\end{equation}
which along with \eqref{E:parameter-2} completes the proof. 
\end{proof}

Therefore, there exists $h^{cs}\in \Gamma_{\mu,\delta}$ such that $\mathcal{T}h^{cs}=h^{cs}$.  Let $\wt{x}^{cs}(t)$ be the solution to \eqref{lp} with $h=h^{cs}$ and let $\wt{a}^+(t)=h^{cs}(\wt{x}^{cs}(t))$. Using the definition of $h^{cs}$, one has 
\begin{equation}
\begin{split}
\wt{a}^+(t)=&-\int_{0}^{\infty}e^{-\lambda_{c}s}\widetilde{G}^{+}\left(\wt{x}^{cs}(t+s), \wt{a}^+(t+s)\right)ds\\
=&-\int_{t}^{\infty}e^{-\lambda_{c}(s-t)}\widetilde{G}^{+}\left(\wt{x}^{cs}(s), \wt{a}^+(s)\right)ds\\
=&-e^{\lambda_{c}t}\int_{0}^{\infty} e^{-\lambda_c s} \widetilde{G}^{+}\left(\wt{x}^{cs}(s), \wt{a}^+(s)\right)ds+\int_{0}^{t}e^{\lambda_{c}(t-s)}\widetilde{G}^{+}\left(\wt{x}^{cs}(s), \wt{a}^+(s)\right)ds\\
=&e^{\lambda_{c}t}\wt{a}^+(0)+\int_{0}^{t}e^{\lambda_{c}(t-s)}\widetilde{G}^{+}\left(\wt{x}^{cs}(s), \wt{a}^+(s)\right)ds,
\end{split}
\end{equation}
which implies that $(\wt{x}^{cs}(t),h^{cs}(\wt{x}^{cs}(t))$ is a solution to \eqref{E:cut}. As mentioned in Section \ref{outline}, the graph of $a^+=h^{cs}(x^{cs})$ over $X^{cs}(\delta)$ is the center-stable manifold, i.e.,
\begin{equation}
\CW^{cs}(\CM_{c})=\{ \Phi\big(x^{cs}, a^+=h^{cs}(x^{cs})\big)\mid x^{cs}\in B^{1}(\delta)\oplus\CX_c^{e}(\delta)\}. 
\end{equation} 

Together with Lemma \ref{L:reduce} and Remark \ref{R:C-0}, we have the local invariance of $\CW^{cs}(\CM_c)$ under \eqref{eqintf}. Recall the coordinate mapping $\Phi$ defined in \eqref{E:coord-1}. 

\begin{theorem}\label{centerstable}
If the solution $U(t)= \Phi\big( y(t), a^+(t), a^-(t), V^e(t)\big)$ to \eqref{eqintf} satisfies $|a^\pm(t)| \le \delta$ and $V^e(t)\in \CX^{e}(\delta)$  for $t\in [0,T]$ with $T>0$ and $U(0)\in \CW^{cs}(\CM_c)$, then $U(t)\in \CW^{cs}(\CM_c)$ for $t\in [0,T]$.
\end{theorem}

\begin{remark}
Later, we will prove the orbital stability on and a characterization of  the center-stable manifold, which yields the local uniqueness of the center-stable manifold. Therefore, we can patch the center-stable manifold of all the solitary waves together to form the center-stable manifold of $\mathcal{M}$.  
\end{remark}

\subsection{Construction of local center-unstable manifolds and center manifolds}
Denote 
 \begin{align*} 
 & X^{cu}(\delta)=\{(y,a^+, V): \Vert V\Vert_{H^1}\le \delta\},\qquad x^{cu}=(y,a^+,V),\\
 &\widetilde{G}^{cu} (x^{cu}, a^-) = (\widetilde{G}^T, \widetilde{G}^-, \widetilde{G}^e) (x^{cu}, a^-),  \\
&A^{cu} (y, \tilde y) = diag\big( 0, -\lambda_{c}, \Pi_{c,y}^{e}JL_{c,y}^e \Pi_{c,y}^{e}+\tilde y \mathcal{F} (c, y)  \big).
\end{align*}

We shall consider the following system of $x^{cs}$ and $a^+$, 
\begin{subequations} \label{E:cut-2}
\begin{equation} \label{E:cut-cs-1}  
\p_t x^{cu} = A^{cu}\big(y, \widetilde{G}^T (x^{cu}, a^-) \big) x^{cu} + \widetilde{G}^{cu}(x^{cu}, a^-)
\end{equation} 
\begin{equation} \label{E:cut-u-1}
\p_t a^- =- \lambda_{c} a^- + \widetilde{G}^- (x^{cu}, a^-).
\end{equation}
\end{subequations}

Define
\begin{equation}  
\Gamma_{\mu, \delta} = \{h : X^{cu} (\delta)  \to \R  \mid  h(y, 0,0) =0, \;  \Vert h\Vert_{C^0} \le \delta,  Lip(h)_{\Vert \cdot \Vert_{H^1, A}} \le \mu\}.
\end{equation}
For any $h \in \Gamma_{\mu, \delta}$ and $\bar x^{cu}\in X^{cu}(\delta)$, let $x^{cu}(t) = (y, a^+, V)(t) \in X^{cu}$ be the backward solution to 
\begin{equation}
\p_t x^{cu} = A^{cu} \big(y, \widetilde{G}^T(x^{cu}, h(x^{cu})\big) \big)x^{cu} + \widetilde{G}^{cu} \big(x^{cu}, h(x^{cu}) \big), \qquad x^{cu}(0)= \bar x^{cu}. 
\end{equation}
Then we define $\widetilde{h}(\bar{x}^{cu})$ as 
\begin{equation}
\widetilde{h} (\bar{x}^{cu})
= \bar a^-=\int_{-\infty}^{0}e^{-\lambda_{c}^- s}\widetilde{G}^{-} \big(x^{cu}(s), h(x^{cu}(s))\big) ds.
\end{equation}

Under suitable assumptions on $A$, $\delta$ and $\mu$ , $\wt h \in \Gamma_{\mu, \delta}$ is well-defined and the transformation $h\rightarrow \wt h$ is a contraction on $\Gamma_{\mu, \delta}$. The graph of the unique fixed point, restricted to the set $B^{1} (\delta) \oplus \CX_c^e (\delta)$
would be the desired center-unstable manifold $\CW^{cu}(\CM_c)$. Similar to the center-stable case,
we have the following theorem. 
\begin{theorem}\label{centerunstable}
If the solution $U(t)= \Phi\big( y(t), a^+(t), a^-(t), V^e(t)\big)$ to \eqref{eqintf} satisfies $|a^\pm(t)| \le \delta$ and $V^e(t)\in \CX^{e}(\delta)$  for $t\in [-T, 0]$ with $T>0$ and $U(0)\in \CW^{cu}(\CM_c)$, then $U(t)\in \CW^{cu}(\CM_c)$ for $t\in [-T,0]$.
\end{theorem}

We obtain the local center manifold $\CW^c(\CM_c)$ as the intersection of the center-stable and center-unstable manifolds. In fact a point $U=\Phi(y,a^+,a^-,V^e)\in \CW^{c}(\CM_c)$ if and only if 
\begin{equation}\label{fpcm}
\begin{cases}
a^+=h^{cs}\left(y,a^{-},V\right)\\
a^-=h^{cu}\left(y,a^{+},V\right).
\end{cases}
\end{equation}
Since the Lipschtiz constant of $h^{cs}$ and $h^{cu}$ are both $\mu<\frac 1 2$, fixing $y$ and $V\in H^{1}$ with $\Vert V \Vert_{H^1}\le \delta$,  $(h^{cu},h^{cs})$ is a contraction with Lipschitz constant $\mu$ on $\R^2$, and consequently, it has a fixed point $(a^+,a^-)=h^{c}\left(y,V\right)$. Clearly $h^{c}\left(y,V\right)$ has a Lipschitz constant $\frac \mu{1-\mu}$ in the $\Vert \cdot \Vert_{H^1, A}$ norm. The graph of $(a^+,a^-)=h^c(y,V)$ restricted to $\CX_c^e(\delta)$ is the desired center manifold. 

\begin{theorem}\label{center}
If the solution $U(t)= \Phi\big( y(t), a^+(t), a^-(t), V^e(t)\big)$ to \eqref{eqintf} satisfies $|a^\pm(t)| \le \delta$ and $V^e(t)\in \CX^{e}(\delta)$  for $t\in [-T, T]$ with $T>0$ and $U(0)\in \CW^{c}(\CM_c)$, then $U(t)\in \CW^{c}(\CM_c)$ for $t\in [-T,T]$.
\end{theorem}

\section{Smoothness of Center-stable manifolds}\label{smoothness of invariant manifolds}

In this section, assuming \eqref{E:parameter-1}, \eqref{E:parameter-2}, \eqref{parameter-set-2}, and 
\begin{equation} \label{E:parameter-3}
C\delta (\lambda_c - C \eta)^{-1} < \eta.
\end{equation}
we prove the smoothness of the center-stable manifold $\CW^{cs}(\CM_c)$ with respect to $(y,a^-,V)$, the smoothness of the center-unstable manifold can be proved similarly. Then one automatically obtains the smoothness of the center manifold since it is the intersection of the center-stable and center-unstable manifolds.  The smoothness of the local invariant manifolds  with respect to $c$ will be discussed in Section \ref{classification}.

Despite the substantial difference in estimates, the proof of the smoothness fits in the framework in \cite{JLZ}, where smooth local invariant manifolds of traveling waves of the Gross-Pitaevskii equation were constructed. With all the estimates established in Section \ref{LinearAnalysis} and Section \ref{construction}, actually the proof is quite similar to the one in \cite{JLZ}. We will sketch the main steps of proving the $C^1$ smoothness. Following the approach in \cite{JLZ}, one may prove higher order smoothness. Our proof of $C^1$ smoothness here illustrates how to adapt the estimates for gKDV to fit in the framework in \cite{JLZ}.

For the simplicity of the presentations, we first  have to introduce some notations. For $t \ge 0$, let 
\[
\Psi(t, x^{cs})= (y(t), a^-(t), V(t)), \quad x^{cs}  \in X^{cs} (\delta), 
\]
be the solution to \eqref{lp} with $h=h^{cs}$ and initial value $x^{cs}$.  By Lemma \ref{L:reduce}, we have 
\begin{equation} \label{E:reduce-1}
\wt \Pi^e \Psi (t, x^{cs})  = \Psi(t, x^{cs}), \quad \forall t\ge 0 \; \; \text{ if } \; \wt \Pi^e x^{cs}= x^{cs}. 
\end{equation}
Moreover, assuming \eqref{E:parameter-1}, \eqref{E:parameter-2}, and \eqref{parameter-set-2}, Lemma \ref{aprior3} and \ref{L:WP1} imply, for all $t \ge 0$,  
\begin{equation} \label{E:Lip-Psi}
Lip_{\Vert \cdot \Vert_{H^1, A}} \Psi(t, \cdot) \le C  e^{C\eta t},
\quad \Psi(t, x^{cs}) \in X^{cs} (C\delta), \; \forall x^{cs} \in X^{cs}(\delta).
\end{equation}

We first outline our approach of proving the $C^1$ smoothness briefly.  As the fixed point of the transformation $\CT$, $h^{cs}$  satisfies
\begin{equation} \label{E:fix-p} 
h^{cs} (x^{cs}) = -\int_{0}^\infty e^{-\lambda_{c} t} \wt G^+\Big(\Psi(t, x^{cs}), h^{cs}\big(\Psi(t, x^{cs})\big) \Big) dt.  
\end{equation}
Since \eqref{lp} is autonomous, a time translation of \eqref{E:fix-p} implies, for $t\ge 0$, 
\begin{equation} \label{E:fix-p-1} 
h^{cs}\big(\Psi(t, x^{cs})\big) = -\int_{t}^\infty e^{ \lambda_c(t-\tau)} \wt G^+\Big(\Psi(\tau, x^{cs}), h^{cs}\big(\Psi(\tau, x^{cs})\big) \Big) d\tau.  
\end{equation}
Differentiating \eqref{E:fix-p} formally,  we obtain, for any $\wt W \in X^{cs}$, 
\begin{align*}
Dh^{cs} (x^{cs}) \wt W =&  -\int_{0}^\infty e^{-\lambda_c t} \Big(D_{a^+}  \wt G^+\big(\Psi(t, x^{cs}), h^{cs}\big(\Psi(t, x^{cs})\big) \big)Dh^{cs} \big(\Psi(t, x^{cs})\big) \\
&+ D_{x^{cs}} \wt G^+\big(\Psi(t,x^{cs} ), h^{cs}\big(\Psi(t, x^{cs})\big) \big)\Big) D\Psi(t, x^{cs}) \wt W dt. 
\end{align*}
Here $D\Psi$ also depends on $Dh^{cs}$ as it solves the following system of equation derived by differentiating \eqref{lp}
\begin{equation} \label{E:DPsi}
\p_t D\Psi = A^{cs} \big(y(t), \wt G^T \big) D\Psi + \CG_1 (\Psi) D\Psi + \wt\CG_1 (\Psi) Dh^{cs} D\Psi,  
\end{equation} 
where $\Psi$ and $D\Psi$ are evaluated at $(t, x^{cs})$, $\wt G^{cs}$ at $( \Psi, h^{cs})$, $h^{cs}$ and $Dh^{cs}$ at $\Psi$. In the above $\CG_1 \in C^m \big(X^{cs},  \mathcal{L}(X^{cs})\big)$ and $\wt \CG_1 \in C^m (X^{cs}, X^{cs})$ are given by 
\begin{equation} \label{E:TCG-1} \begin{split}
\wt \CG_1 (x^{cs})  =& D_{a^+} \big(A^{cs} \big(y, G^T(x^{cs},a^+) \big) \big)  x^{cs}+ D_{a^+} \wt G^{cs}\\
=& \big(0, 0, (D_{a^+} \wt G^T) \CF (c,y) V \big)+ D_{a^+} \wt G^{cs},\\
\end{split} \end{equation}
\begin{equation} \label{E:CG-1} \begin{split}
\CG_1 (x^{cs}) \wt W = & D_{x^{cs}} \big(A^{cs} \big(y, G^T(x^{cs}, a^+) \big) \big) (\wt W) x^{cs}+ D_{x^{cs}} \wt G^{cs} (\wt W)\\
=& \Big(0, 0, \wt y \big( D_y A^e(y) \big) V+  \big( D_{x^{cs}} \wt G^T (\wt W)\big) \CF\big( c, y \big)V  \\ 
&\qquad \qquad \qquad + \wt y \wt G^TD_y \CF(c, y) V\Big)  + D_{x^{cs}} \wt G^{cs} (\wt W) 
\end{split} \end{equation}
where $x^{cs} =(y, a^-, V)$, $\wt W =(\wt y, \wt a^-,\wt V) \in X^{cs}$, $a^+$ is evaluated at $h^{cs} (x^{cs})$, and $\wt G^{cs}$ is evaluated at $\big(x^{cs}, h^{cs}(x^{cs})\big)$. 

Denote
\[
Y_1= C^0\big(X^{cs} (\delta), \mathcal{L}(X^{cs}, \R) \big).
\]
Inspired by the above formally derivation, we define a linear transformation $\CT_1$ on 
\[
Y_1(\mu)= \{ \CH\in Y_1, \Vert \CH\Vert_{Y_1}\le \mu \big) 
\]
as, for any $\CH \in Y_1(\mu)$, $x^{cs}\in X^{cs} (\delta)$, and $\wt W \in X^{cs}$, 
\begin{equation} \label{E:CT1} \begin{split}
(\CT_1 \CH) (x^{cs}) \wt W =&   -\int_{0}^\infty e^{-t\lambda_c} \Big(D_{x^{cs}} \wt G^+\big(\Psi, h^{cs} (\Psi) \big)\\
&+D_{a^+}  \wt G^+\big(\Psi, h^{cs}(\Psi) \big)\CH \big(\Psi\big) \Big) \Psi_1 (t) \wt W dt
\end{split} \end{equation}
where $\Psi$ is evaluated at $(t, x^{cs})$. Operator $\Psi_1 (t) \in \mathcal L(X^{cs})$ satisfies $\Psi_1(0) =I$ and 
\begin{equation} \label{E:Psi-1}
\p_t \Psi_1 = A^{cs} \big(y(t), \wt G^T \big) \Psi_1 + \CG_1 (\Psi) \Psi_1 + \big(\CH (\Psi) \Psi_1 (\cdot)\big)\wt \CG_1 (\Psi),   
\end{equation}
where $\CG$ and $\CG_1$ are given in \eqref{E:CG-1}, $\wt G^{cu}$ is evaluated at $\big( \Psi, h^{cs} (\Psi)\big)$, and $\CH$ at $\Psi(t, x^{cs})$. 
Note that $h^{cs}\in \Gamma_{\mu,\delta}$, it is natural to require the $\Vert Dh^{cs}\Vert_{Y_1}\le \mu.$ Just as in Remark \ref{R:C-0}, the right side of \eqref{E:Psi-1} and the integrand in \eqref{E:CT1} are well-defined. Since \eqref{lp} is autonomous, when $x^{cs}$ is shifted to $\Psi(t_0, x^{cs})$, the principle fundamental solution to the associated \eqref{E:Psi-1} becomes $\Psi_1(t+t_0) \Psi_1(t_0)^{-1}$. Therefore we obtain 
\begin{equation} \label{E:CT1-shift} \begin{split}
(\CT_1 \CH) \big( \Psi(t_0, x^{cs})\big) \Psi_1(t_0) & \wt W =   -\int_{t_0}^{\infty} e^{(t_0-t)\lambda_c^+} \Big(D_{x^{cs}} \wt G^+\big(\Psi, h^{cs} (\Psi) \big)\\
&+D_{a+}  \wt G^+\big(\Psi, h^{cs}(\Psi) \big)\CH \big(\Psi\big) \Big) \Psi_1 (t) \wt W dt, 
\end{split} \end{equation}
where $\Psi$ is still evaluated at $(t, x^{cs})$ and $\Psi_1$ defined for $x^{cs}$. 

If $h^{cs} \in C^1$, then $Dh^{cs}$ must be the fixed point of $\CT_1$. Therefore, our strategy to prove $h^{cs} \in C^1$ is to show 
\begin{enumerate}
\item $\CT_1$ is a well-defined contraction on $Y_1(\mu)$,
\item  the fixed point of $\CT_1$ is indeed $Dh^{cu}$. 
\end{enumerate}
Throughout the procedure, \eqref{E:parameter-1}, \eqref{E:parameter-2}, and \eqref{parameter-set-2} are assumed. 

{\it Step 1: show $\CT_1 \CH \in Y_1(\mu) $. } 
Analogous to Lemma \ref{L:WP1}, we have that for any $x^{cs}\in X^{cs}(\delta)$,
\begin{equation}\label{estimate-psi1}
\Vert\Psi_1(t,x^{cs})\wt W\Vert_{H^1,A}\le Ce^{C\eta t}\Vert \wt W\Vert_{H^{1},A},
\end{equation}
which along with \eqref{parameter-set-2} implies
\begin{equation}\label{C0-H}
\vert(\CT_1 \CH) (x^{cs}) \wt W\vert \le C\delta (\lambda_{c}-C\eta)\Vert \wt W\Vert_{H^{1},A} \le \mu \Vert \wt W\Vert_{H^{1},A}.
\end{equation}

Much as \eqref{C0-H}, it also holds that $\CT_1^{(n)} (\CH) \to \CT_1(\CH)$ uniformly in $x^{cs}$, where 
\begin{align*} 
\big(\CT_1^{(n)} ( \CH)\big) (x^{cs}) \wt W =&   -\int_{0}^n e^{-\lambda_c t} \Big(D_{x^{cs}} \wt G^+\big(\Psi, h^{cs} (\Psi) \big)\\
&+D_{a^+}  \wt G^+\big(\Psi, h^{cs}(\Psi) \big)\CH \big(\Psi\big) \Big) \Psi_1 (t) \wt W dt.
\end{align*} 
From the continuity of $D\wt G^{cs, +}$, it is easy to verify that $\big(\CT_1^{(n)} (\CH)\big) (x^{cs})$ is $C^0$ in $x^{cs}$. Therefore $\CT_1(\CH)$ is also continuous and thus $\CT_1(\CH) \in Y_1 (\mu)$.  

{\it Step 2: estimate the Lipschtiz constant of $\CT_1$.} Let $\CH_j\in Y_1(\mu)$ and $\Psi_{1, j}(t)$ be defined in  \eqref{E:Psi-1} for $\CH_j$, $j=1,2$, which satisfy 
\begin{align*}
\p_t (\Psi_{1,2} - \Psi_{1,1}) = &\big( A^{cs} (c,y, G^T) - \CG_1(\Psi) - \wt \CG_1(\Psi) \CH_1 \big) (\Psi_{1,2} - \Psi_{1,1})\\
& + \big((\CH_2 - \CH_1)  (\Psi) \Psi_{1,2}\big) \wt \CG_1(\Psi)
\end{align*}
and $(\Psi_{1,2} - \Psi_{1,1})(0)=0$. 

From estimate \eqref{estimate-psi1} on homogeneous solutions to \eqref{E:Psi-1} and the variation of constant formula, we obtain 
\[
\Vert (\Psi_{1,2}(t,x^{cs}) - \Psi_{1,1}(t,x^{cs}))\wt W \Vert_{H^1,A} \le C\delta t e^{C\eta t} \Vert \CH_2 - \CH_1\Vert_{Y_1}\Vert \wt W\Vert_{H^1,A}
\]
where we also used $\Vert \wt \CG_1 \Vert_{H^1, A} \le C \delta$ which is obvious from its definition. 
According to the definition of $\CT_1$, we have, for any $ x^{cs}\in X^{cs}(\delta)$,  
\begin{align*}
\big(\CT_1(\CH_1) - \CT_1(\CH_2)\big) (x^{cs}) =& -\int_{0}^\infty e^{-\lambda_{c} t} \big(D_{a^+} \wt G^+ (\CH_2 - \CH_1) \Psi_{1,2} (t) \\
&+(D_{x^{cs}} \wt G^+ + D_{a^+} G^+ \CH_1) (\Psi_{1,2} - \Psi_{1,1}) (t)  \big)  dt,   
\end{align*}
where $D\wt G^+$ is evaluated at $\big(\Psi, h^{cs} (\Psi)\big)$, $\CH_j$ at $\Psi$, and $\Psi$ at $(t, x^{cs})$. Using \eqref{estimate-psi1}, and the above estimates on  $\Psi_{1,2} - \Psi_{1,1}$, it follows that 
\begin{align*}
\Vert \CT_1(\CH_1) - \CT_1(\CH_2) \Vert_{Y_1} \le C\delta  (\lambda -C\eta)^{-2} \Vert \CH_2 - \CH_1\Vert_{Y_1}.
\end{align*}

Assume 
\begin{equation} \label{E:parameter-5}
C \delta  (\lambda -C\eta)^{-2} <1,
\end{equation}
then $\CT_1$ is a contraction mapping on $Y_1(\mu)$. Let $\CH^{cs} \in Y_1(\mu)$ be the unique fixed point of $\CT_1$. 

{\it Step 3: Show $Dh^{cs}=\CH^{cs}$.} Since $\CH^{cs}(x^{cs})$ is continuous in $x^{cs}$, it suffices to show $Dh^{cs} (x^{cs}_0) \wt W = \CH^{cs}(x^{cs}_0) \wt W$ at any fixed $x^{cs}_0 \in X^{cs}(\delta)$ and $\wt W \in \wt X^{cs}\backslash \{0\}$. 
Let $\Psi_1(t)$ be defined as in \eqref{E:Psi-1} associated to $\CH^{cs}$ and $x^{cs}_0$ and 
\begin{align*}
&R_\Psi (t) = \Psi(t, x^{cs}_0 +\wt W) - \Psi(t, x^{cs}_0) - \Psi_1 (t) \wt W, \\
&R_h (t) = h^{cs} \big( \Psi(t, x^{cs}_0 + \wt W) \big) - h^{cs} \big( \Psi(t, x^{cs}_0)\big) - \CH^{cs} \big(\Psi(t, x^{cs}_0)\big) \Psi_1(t) \wt W. 
\end{align*} 
Denote \begin{align*} 
&W(s,t) = (1-s) \Psi(t, x^{cs}_0) + s \Psi(t, x^{cs}_0 + \wt W), \\ 
&a^+ (s,t) = (1-s) h^{cs} \big(\Psi(t, x^{cs}_0)\big)  + s h^{cs} \big( \Psi(t, x^{cs}_0 + \wt W) \big).
\end{align*}
and for $\alpha = cs, +$  
\begin{align*}
R^{\alpha} (t) =& \wt G^{\alpha}  \big( W(1, t), a^+(1, t)\big) - \big[ \wt G^{\alpha}   + D_{x^{cs}} \wt G^{\alpha}  \big( W(1, t) - W(0, t)\big) \\
&+ D_{a^+}\wt G^{\alpha}  \big( a^+(1,t) - a^+(0, t)  \big)  \big]
\end{align*}
where $\wt G^{cs} $ and $D\wt G^{cs} $ in the brackets $[\ldots]$ are evaluated at 
$\big( W(0, t), a^+(0, t)\big) =\Big( \Psi(t, x^{cs}_0), h^{cu}\big(\Psi(t, x^{cs}_0)\big)\Big)$.
From \eqref{E:fix-p} and $\CT_1(\CH^{cs}) = \CH^{cs}$, we have 
\[
R_h(0) = -\int_{0}^\infty e^{-\lambda_c t}\big( R^+ (t) + D_{x^{cs}} \wt G^+ R_\Psi  (t)  + D_{a^+}\wt G^+ R_h(t) \big) dt .
\]
Moreover, using  \eqref{E:fix-p} and \eqref{E:CT1-shift}, we also obtain 
\begin{equation} \label{E:R_h}
R_h(t) = -\int_{0}^\infty e^{-\lambda_c \tau} ( R^+  + D_{x^{cs}} \wt G^+ R_\Psi   + D_{a^+}\wt G^+ R_h)(t+\tau) d\tau,  \quad t\ge 0, 
\end{equation} 
where again the above $D \wt G^+$ are evaluated at $\Big( \Psi(t+\tau, x^{cs}_0), h^{cs}\big(\Psi(t+\tau, x^{cs}_0)\big)\Big)$. 

From \eqref{lp} and \eqref{E:Psi-1}, $R_\Psi(t)$ satisfies $R_\Psi(0)=0$ and 
\begin{align*} 
\p_t R_\Psi& =  \wt A_0^{cs} (t) R_\Psi + \wt A_0^+ (t) R_h + R^{cs} + D_{x^{cs}} \wt G^{cs} R_\Psi  + D_{a^+}\wt G^{cs} R_h+\int_0^1 (\wt A_s^{cs} \\
&- \wt A_0^{cs}) (t) \big( W(1,t) - W(0, t) \big)  
+ (\wt A_s^+ - \wt A_0^+) (t) \big( a^+(1,t) - a^+(0, t)    
\big) ds
\end{align*}
where $D\wt G^{cs}$ is evaluated at $\big( W(0, t), a^+(0, t)\big)$, $\wt A_s^+(t) \in X^{cs}$ and the operator $\wt A_s^{cs}(t) \in L(X^{cs})$ are given by 
\begin{align*}
\wt A_s^{cs}(t) \wt W = & D_{x^{cs}} \Big(A^{cs} \big(y, G^T (x^{cs},a^+)\big) x^{cs}\Big)|_{\big(W(s,t), a^+(s,t)\big)} (\wt W)\\
=& A^{cs} \Big(y(s, t), G^T\big(W(s,t), a^+(s,t) \big)\Big) \wt W \\
&+ D_{x^{cs}} \Big(A^{cs} \big(y, G^T (x^{cs},a^+)\big)\Big)|_{\big(W(s,t), a^+(s,t)\big)} (\wt W) W(s,t)  \\
\wt A_s^+(t) = &D_{a^+} \Big(A^{cs} \big(y, G^T (x^{cs},a^+)\big)\Big)\vert_{\big(W(s, t), a^+(s, t)\big)} W(s,t)
\end{align*}
with $W(s,t)$ and $a^+(s,t)$ defined in the above and $y(s,t)$ being the $y$ component of $W(s,t)$ (so the $D_{x^{cs}}$ also acts on the $y$ component in $A^{cs}$). 
 
For $\wt W = (\wt y, \wt a^-, \wt V)$, we have 
\[
\wt A_s^{+}  \wt W= (0, 0,  V^+ ), \quad (\wt A_s^{cs} - A^{cs}|_{(W(s, t), a^+(s, t))}) \wt W= \big(0, 0, V^{cs} )
\]
and from Lemma \ref{lipGe}, 
\begin{equation} \label{E:temp-2.9}
\Vert V^{cs, +} \Vert_{L_{[t_0, t_0+T]}^1 H^1} \le C T^{\frac 12} (1+T) \Vert W \Vert_{X_{[t_0, t_0+T]}^{cs} }\Vert \wt V \Vert_{L_{[t_0, t_0+T]}^\infty H^1}. 
\end{equation} 

We first consider $t_2 \in (t_1, t_1+1)$ in the following estimates, where we can use $W\in X_{[t_1, t_2]}^{cs} (C\delta)$ due to Lemma \ref{L:WP1} which also yields 
\begin{equation}\label{W1W0-ST} \begin{split}
\Vert W(1,t)- W(0,t)\Vert_{ST^{c}_{[t_1,t_2]}, A}\le& 
C \Vert W(1,t_1)-W(0,t_1)\Vert_{H^1,A}
\le C e^{C \eta t_2}\Vert \wt W\Vert_{H^1,A}.
\end{split}\end{equation}
A similar argument would imply 
\begin{equation}
\Vert \Psi_1(t)\wt W\Vert_{ST^{c}_{[t_1,t_2]}, A}\le 
C e^{C \eta t_2}\Vert \wt W\Vert_{H^1,A}.
\end{equation}
Inequality \eqref{W1W0-ST} along with  \eqref{lipG} and Lemma  \ref{lipGe} implies
\begin{equation}  \label{E:R_cs} 
\begin{split}
\Vert R^{cs}(t)\Vert_{L^1_{[t_1,t_2]}X^{cs}} + |R^+(t_2)| \le & 
C \Vert W(1,t)- W(0,t)\Vert^2_{X^{cs}_{[t_1,t_2]}} 
\le  C e^{C\eta t_2} \Vert \wt W \Vert_{H_1, A}^2.
\end{split}
\end{equation}
for any $0\le t_1<t_2$. 

The integral terms in $\p_t R_\Psi$ can be estimated by Lemma \ref{lipGe}, which along with inequalities \eqref{W1W0-ST} and \eqref{E:R_cs} implies 
\begin{align*} 
&\Vert \p_t R_\Psi - \wt A_0^{cs}(t) R_\psi - \wt A_0^+ (t) R_h - D_{x^{cs}} \wt G^{cs} R_\Psi  - D_{a^+}\wt G^{cs} R_h  \Vert_{L^1_{[t_1,t_2]}X^{cs}} \\ 
\le & 
C \Vert W(1,t)- W(0,t)\Vert^2_{X^{cs}_{[t_1,t_2]}} \le  C  e^{C\eta t_2} \Vert \wt W \Vert_{H_1, A}^2.
\end{align*}
 
Again we apply Lemma \ref{lipGe} and inequalities \eqref{lipG} and \eqref{E:temp-2.9} to estimate other remainder terms linear in $R_\Psi$ and $R_h$ and obtain 
\begin{align*}
&\Vert \p_t R_\Psi - A^{cs}\left(y(0,t),G^T(W(0,t),a^{+}(0,t))\right) R_\psi  \Vert_{L^1_{[t_1,t_2]}(H^1,A)} \\ 
\le &C\eta (\Vert R_\psi \Vert_{ST^{c}_{[t_1,t_2]},A} + |R_h|_{L^{\infty}_{[t_1,t_2]}})  + C e^{C\eta t_2}  \Vert \wt W \Vert_{H^1, A}^2.
\end{align*}

With the above estimates, following the same arguments in the proof of Lemma \ref{L:WP1}, we have 

\begin{equation} \label{E:R-Psi} \begin{split}
\Vert R_\Psi (t) \Vert_{H_1, A} \le & C\big( \|e^{C \eta (t-\cdot)}R_h\|_{L^1_{[0,t]}}+  e^{C\eta t}  \Vert \wt W\Vert_{H_1, A}^2\big)
\le  Ce^{C \eta t}( \eta^{-1}  R_{h,\infty} + \Vert \wt W\Vert_{H_1, A}^2\big)
\end{split}\end{equation}
where $R_\psi(0) =0$ was used and 
\[
R_{h, \infty} \triangleq \Vert e^{-C\eta t} R_h(t) \Vert_{L_{\R^+}^\infty}. 
\]
Here $R_{C, \infty}<\infty$ for some $C>0$ is due to \eqref{E:Lip-Psi}  and \eqref{estimate-psi1}. 
Substituting this into \eqref{E:R_h}, using \eqref{lipG}, and noting that the estimate on $R_+(t_2)$ in \eqref{E:R_cs} is independent of $t_1$, we can compute 
\begin{align*}
R_{h, \infty} \le &Ce^{-C \eta t} \int_0^\infty e^{-\lambda_c \tau} \big( e^{C \eta (t+\tau)}\Vert \wt W\Vert_{H_1, A}^2 + \delta ( |R_h|+ \Vert R_\Psi \Vert_{X^{cs}} ) (t +\tau)\big) d\tau \\
\le & Ce^{-C \eta t} \int_0^\infty e^{- \lambda_c \tau} e^{C \eta (t+\tau)}( \Vert \wt W\Vert_{H_1, A}^2 + \delta \eta^{-1} R_{h, \infty} ) d\tau \\ 
\le& C (\lambda_c - C\eta)^{-1}  (\Vert \wt W\Vert_{H_1, A}^2 +\delta \eta^{-1} R_{h, \infty}). 
\end{align*}
Therefore assumption \eqref{E:parameter-3} implies 
\[
R_{h, \infty} \le  C (\lambda_c - C\eta)^{-1} \Vert \wt W\Vert_{H_1, A}^2. 
\]
By letting $t=0$, we have $|R_h(0)| \le  C (\lambda_c - C\eta)^{-1} \Vert \wt W\Vert_{H_1, A}^2$ which  complete the proof of $C^1$ smoothness of the center-stable manifold.

Finally, we prove the center-stable manifold is tangent to the center-stable subspace along $\CM_c$. 
\begin{lemma} \label{L:tangency} 
There exists $C>0$ such that if $A$ and $\delta$ satisfy \eqref{E:parameter-1}, \eqref{E:parameter-2},  \eqref{parameter-set-2}, \eqref{E:parameter-5},  and \eqref{E:parameter-3}, we have $Dh^{cs}(y, 0, 0, 0) =0$.
\end{lemma}

\begin{proof}
Observe that \eqref{lp} and the definition of $\wt G^{cs}$ implies $\Psi\big(t, (y, 0,0)\big) =(y, 0, 0)$ for all $t\ge 0$. 
For any $\CH \in Y_1$, \eqref{newh}, the fact $D \wt G^+(y,0,0,0)=0$, and the above observation imply $\CT_1 (\CH) (y,0,0, 0) =0$. Therefore $Dh^{cs}(y, 0, 0, 0) =0$ at any $y \in \R$, which implies that at any solitary wave on $\CM_c$, the center-stable manifold is tangent to the center-stable subspace.
\end{proof}

\section{Local Dynamics near solitary waves}
In this section, we study the local dynamics near solitary waves based on local invariant manifolds. We will prove: (i) the center-stable manifold repels nearby orbits in positive time and attracts nearby orbits in negative time; (ii) on the center-stable manifold, center manifold attracts nearby orbits in positive time; and (iii) the orbital stability on center manifolds. Various norms in the below are defined in Sections \ref{LinearAnalysis} and \ref{construction}. Even though we are still working with the modified system \eqref{E:cut-cs} and \eqref{E:cut-u}, by taking $\delta>0$ much smaller than the one in the cut-off, all the results valid in a $C\delta$-neighborhood in this section hold for the original gKDV equation.

\subsection{Dynamics near the center-stable and center-unstable manifolds}\label{classification}
In this  subsection, we study the local dynamics for initial data near the center-stable manifold. 

\begin{proposition}\label{offcs}
Let $U(t)=Em^\perp\big(x^{cs}(t),a^+(t)\big)\in Em^\perp \left(B^1(\delta)\oplus B^1(\delta)\oplus \CX_c^e(\delta)\right)$ be a solution to \eqref{eqintf}  for $t\in [0,T]$ with the initial data $U(0)=Em^\perp(\bar x^{cs}, \bar{a}^{+})$. 
 
We have 
$$\left\vert a^{+}(t)-h^{cs}\big(x^{cs}(t)\big)\right\vert \ge e^{(\lambda_{c}-C\delta) t}\left\vert \bar{a}^{+}-h^{cs}(\bar x^{cs})\right\vert \qquad\forall t\in [0,T],$$
\end{proposition}

The above inequality indicates that the center-stable manifold repels nearby orbits forward in time and, as \eqref{eqintf} is autonomous, it also attracts nearby orbits backward in time. 

\begin{proof}
Without loss of generality, we may assume $\bar a^+ \ne h^{cs}(\bar x^{cs})$. Let $\wt x^{cs}(t)$ be the solution to $\eqref{lp}$ with $h=h^{cs}$ and $\wt x^{cs}(t_0)=x^{cs}(t_0)$, and let $\wt{a}^+= h^{cs}(\wt x^{cs})$. By the invariance of the center-stable manifold, we have 
\begin{equation}\label{diffofa}
\partial_{t}\left(a^{+}-\wt{a}^+\right) =\lambda_{c}\left( a^{+}-\wt{a}^+\right)+
\widetilde{G}^{+}(x^{cs},a^+)-\widetilde{G}^{+}(\wt x^{cs}, \wt{a}^+). 
\end{equation}
Since \eqref{lipG} yields 
\begin{equation}\label{diffofG}
\begin{split}
\left\vert \widetilde{G}^{+}(x^{cs},a^+)-\widetilde{G}^{+}(\wt x^{cs}, \wt{a}^+)\right\vert
\le C\delta \left(\Vert x^{cs}-\wt x^{cs}\Vert_{H^1,A} + \vert a^{+}-\wt{a}^+\vert \right),
\end{split}
\end{equation}

and $x^{cs}(t_0)-\wt x^{cs}(t_0)=0$, \eqref{diffofa} and \eqref{diffofG} yield $\partial_{t} | a^{+}-\wt{a}^+|_{t=t_0}> 0$. Let 
$$T_{1}:=\sup\left\{t\in[0, 1]: \partial_{t}|a^{+}-\wt{a}^+|> 0 \;\text{in}\; [t_{0},t_{0}+t) \right\}.$$ 

We show $T_{1}=1$ in the below. Suppose otherwise $T_{1}<1$, by its definition, one has 
$$\partial_{t}|a^{+}-\wt{a}^+|_{t=t_{0}+T_{1}}=0, \quad \Vert a^{+}-\wt{a}^+\Vert_{L^{\infty}_{[t_{0},t_{0}+T_1]}} = \vert a^{+}(t_{0}+T_1)-\wt{a}^+(t_{0}+T_1)\vert.$$ 
By Lemma \ref{L:WP1}, we have 
\[
\Vert x^{cs}(t_{0}+T_{1})-\wt x^{cs}(t_{0}+T_{1})\Vert_{H^1, A} \le  C \eta | a^{+}(t_{0}+T_{1})-\wt{a}^+(t_{0}+T_{1})|.
\]
 It follows from \eqref{diffofa} and \eqref{diffofG}
$$\partial_{t}| a^{+}-\wt{a}^+ |_{t=t_{0}+T_{1}}>0,$$
which is a contradiction to the definition of $T_{1}$ and $T_1<1$. Therefore, $T_1=1$ and for $t\in [0, 1]$, $\Vert a^+ -\wt a^+\Vert_{L_{[t_0, t_0+t]}^\infty}$ is always achieved at $t_0+t$. Again from Lemma \ref{L:WP1}, we have that, for any $t\in [0,1]$,
\begin{equation}\label{diffoforbits2}
\Vert x^{cs}(t_{0}+t)-\wt x^{cs}(t_{0}+t)\Vert_{H^1, A} \le  C \eta | a^{+}(t_{0}+t)-\wt{a}^+(t_{0}+t)|.
\end{equation}

By applying \eqref{diffofG} and the above inequality to \eqref{diffofa}, we have 
\[
|\partial_{t}(a^{+}-\wt{a}^+) - \lambda_{c}( a^{+}-\wt{a}^+)| \le C \delta | a^{+}-\wt{a}^+|, \quad t\in [0, 1]. 
\]

Then by the Gronwall inequality, we obtain 
$$\vert a^{+}(t_{0}+t)-\wt{a}^+(t_{0}+t)\vert \ge e^{\left(\lambda_{c}-C\delta\right)t}\vert a^{+}(t_{0})-\wt{a}^+(t_{0})\vert, \quad t\in [0, 1].$$
Since Lemma \ref{L:tangency} yields 
\begin{equation} \label{E:D2h}
D h^{cs} \le C\delta \quad \text{ in} \quad Em^\perp \big(B^1(\delta)\oplus B^1(\delta)\oplus \CX_c^e(\delta)\big),
\end{equation}
along with \eqref{diffoforbits2}, the inequality implies, for $t\in [0, 1]$, 

\begin{equation}
\begin{split}
& \left\vert a^{+}(t_{0}+t)-h^{cs}\big(x^{cs}(t_{0}+t)\big)\right\vert \ge (1-C \delta \eta) \vert a^{+}(t_{0}+t)-\wt{a}^+(t_{0}+t)\vert 
\\
\ge &(1-C \delta \eta) e^{(\lambda_{c}-C\delta)t}\vert a^{+}(t_{0})-\wt{a}^+(t_{0})\vert \ge e^{\lambda_{c}-3C\delta}\left\vert a^{+}(t_0)-h^{cs}(x^{cs}(t_0))\right\vert
\end{split}
\end{equation}
Iterating the above estimate, we complete the proof. 
\end{proof} 

\begin{remark}
The exponential type estimate in Proposition \ref{classification} can also be obtained by a more direct approach through considering $\p_t \big( a^{+}-h^{cs}(x^{cs}) \big)$ and using the invariance of $h^{cs}$. Since $\p_t x^{cs} \in H^{-2}$ and $Dh^{cs}$ acts only on $H^1$, this procedure may be carried out for $x^{cs} \in H^4$ and the estimate for $x^{cs} \in H^1$ follows from the continuous dependence in $H^1$ of the solutions on their initial data and the continuity of $h^{cs}$. However, due to the lack of $O(T)$ estimate on $D\wt G^{cs}$ in Lemma \ref{lipGe}, one would only obtain a lower bound in the form of $(1-C\delta) e^{(\lambda_c -C \delta)t}$ and it is not easy to get rid of the factor $1-C\delta$. 
\end{remark}

For any point $U=Em^\perp(y,a^{+},a^{-},V^{e})$ in a small neighborhood of $\CM_c$, the total of the norms $|a^+|+|a^-|+\Vert V^e\Vert_{H^1}$ of its transversal 
components is equivalent to its distance $dist(U, \CM_c)$ to $\CM_c$, where 
\begin{equation} \label{E:dist} 
dist(U, K) = \inf_{\wt U \in K} \Vert U-\wt U\Vert_{H^{1}}
\end{equation} 
for any subset $K \subset H^1$. See Remark \ref{R:metric}. The above Proposition yields the nonlinear instability of the traveling waves with an exit time estimate. 

\begin{corollary}\label{nics}

For any 
$U(0)\notin \CW^{cs}(\mathcal{M}_c)$, 
$\exists T^{\ast}>0
$ such that $$dist\big(U(T^*), \CM_c\big)\ge \delta.$$
\end{corollary}

Parallel to the center-stable case, the center-unstable manifold attracts nearby orbits exponentially as $t \to +\infty$. 

\begin{proposition}\label{offcu}
Let $U(t)=Em^\perp \big(x^{cu}(t),a^-(t)\big)\in Em^\perp \left(B^1(\delta)\oplus B^1(\delta)\oplus \CX_c^e(\delta)\right)$ be a solution to \eqref{eqintf}  for $t\in [0,T]$ with the initial data $U(0)=Em^\perp(\bar x^{cu}, \bar{a}^{-})$. 
We have  
$$\left\vert a^{-}(t)-h^{cu}\big(x^{cu}(t)\big)\right\vert \le e^{-(\lambda_{c} -C\delta) t}\left\vert \bar{a}^{-}-h^{cu}(\bar x^{cu})\right\vert \qquad\forall t\in [0,T].$$
Moreover,  
for any $U(0)\notin \CW^{cu}(\mathcal{M}_c)$, 
$\exists T^{\ast}<0
$ such that $$dist(U(T^*), \CM_c)\ge \delta.$$
\end{proposition}

Since the center manifold is the intersection of the center-stable and center-unstable manifolds, the above theorems imply 

\begin{corollary}\label{nic}
For any $U(0)\notin \CW^{c}(\mathcal{M}_c)$, 
$\exists T^{\ast} \in \mathbb{R}
$ such that $$dist\big(U(t), \CM_c\big)\ge \delta.$$
\end{corollary}

\begin{remark}
Corollary \ref{nic} along with the above exponential estimates indicates that the nonlinear instability of the solitary waves for the supercritical gKDV equations is generic in the sense that if  initial data is not on the co-dim $2$ center manifold, then the flow will leave a neighborhood of the soliton manifold exponentially fast at least in one time direction. This result is stronger than the classical nonlinear instability result 
of the existence of special initial data in any neighborhood of the solitary waves 
whose orbit leaves a neighborhood of the soliton manifold. 
\end{remark}

\subsection{Dynamics inside the center-stable and center-unstable manifolds and the orbital stability inside  center-manifolds}\label{orbital stability}

Based on the exponential estimates in the directions transversal to the center-stable and center-unstable manifolds obtained in Subsection \ref{classification} and the energy conservation, we shall prove the exponential stability of the center manifold inside the center-stable manifold and the orbital stability of the traveling waves inside the center manifold. Recall that the center manifold $\CW^c(\CM_c)$ is the graph $\{a^\pm = h^\pm (y, V^e)\}$ of $h^c = (h^+, h^-)$. Clearly, 
\begin{align} 
&h^+=h^{cs}(y,h^-,V^e), \qquad h^-=h^{cu}(y,h^+,V^e)\label{E:hpm}\\
&h^\pm(y, 0) = 0, \quad D h^\pm (y, 0) =0.  \label{E:hpm1}
\end{align}

\begin{proposition}\label{attract-center}
There exists $C_0\ge 1$ such that the following hold. 
\begin{enumerate}
\item Let $U(t)=Em^\perp(y,a^+,a^-,V^e)(t)$ for $t\ge 0$ be a solution to \eqref{eqintf}  with the initial data 
\[
U(0) = Em^\perp\left(\bar{y},\bar{a}^{+},\bar{a}^{-},\bar{V}^{e}\right)\in \CW^{cs}(\CM_c), \quad \bar V^e \in X_{c, \bar y}^e, \;  |\bar a^\pm|, \Vert \bar V^e \Vert_{H^1} \le C_0^{-1} \delta, 
\]
then 
we have 
$$\left\vert a^{-} -h^- (y,V^e)\right\vert  \le (1+C\delta^2)e^{-(\lambda_{c}-C\delta) t}\left\vert \bar a^- -h^- (\bar y,\bar V^e) \right\vert, \qquad\forall t\ge 0$$
and 
\[
\Vert V^e\Vert_{H^1}^2 \le C (\Vert \bar V^e \Vert_{H^1}^2 + |\bar a^- - h^-(\bar y, \bar V^e)|^3).
\]
\item Let $U(t)=Em^\perp(y,a^+,a^-,V^e)(t)$ for $t\le 0$ be a solution to \eqref{eqintf}  with the initial data 
\[
U(0) = Em^\perp\left(\bar{y},\bar{a}^{+},\bar{a}^{-},\bar{V}^{e}\right)\in \CW^{cu}(\CM_c), \quad \bar V^e \in X_{c, \bar y}^e, \;  |\bar a^\pm|, \Vert \bar V^e \Vert_{H^1} \le C_0^{-1} \delta, 
\]
then 
we have 
$$\left\vert a^{+} -h^+ (y,V^e)\right\vert  \le (1+C\delta^2)e^{-(\lambda_{c}-C\delta) t}\left\vert \bar a^+ -h^+ (\bar y,\bar V^e) \right\vert, \qquad\forall t\ge 0$$
and 
\[
\Vert V^e\Vert_{H^1}^2 \le C (\Vert \bar V^e \Vert_{H^1}^2 + |\bar a^+ - h^+(\bar y, \bar V^e)|^3).
\]
\end{enumerate}
\end{proposition}

\begin{proof}
We will consider the center-stable case, while the other one can be proved similarly. 
Since $U(t) \in \CW^{cs}(\CM_c)$, \eqref{E:hpm} and \eqref{E:D2h} imply that, if $|a^\pm|, \Vert V^e\Vert_{H^1} <\delta$ on $[0, T]$, then for $t \in [0, T]$, 
$$
 \vert h^+(y, V^e) -a^+\vert = \vert h^{cs}(y,h^-,V^e)-h^{cs}(y,a^-,V^e)\vert\le C \delta \vert h^-(y, V^e)-a^- \vert,
$$
and 
\begin{equation*}
\begin{split}
& \vert \big(h^-(y, V^e) -a^-\big) - \big( h^{cu}(y,a^+,V^e)-a^- \big) \vert \le \vert h^{cu}(y,h^+,V^e)-h^{cu}(y,a^+,V^e)\vert \\
 \le & C\delta \vert h^+(y, V^e)-a^+ \vert
  \le  C\delta^2   \vert h^-(y, V^e)-a^- \vert.
\end{split}
\end{equation*}
Applying this inequality at $t \in [0, T]$ and then along with  Proposition \ref{offcu}, we have 
\begin{equation} \label{E:temp-4} \begin{split}
\vert h^-(y, V^e)-a^-\vert\le & (1+ C\delta^2) \vert a^--h^{cu}(y,a^+,V^e)\vert \\
\le& (1+ C\delta^2) e^{-(\lambda_c -C\delta)t} \vert \bar a^--h^{cu}(\bar y,\bar a^+, \bar V^e)|\\
\le & (1+ C\delta^2) e^{-(\lambda_c -C\delta)t} \vert \bar a^--h^- (\bar y, \bar V^e)|.
\end{split} \end{equation}

To estimate the bound on $\Vert V^e\Vert_{H^1}$ on $[0, T]$, let 
\[
\wt E(y,V^e)= (E+cP)\big(Em^\perp (y, a^+, a^-, V^e)\big).
\]
On the one hand, clearly for any $y\in \R$, it holds that 
\begin{align*} 
&(E+cP)'\left(Q_{c}(\cdot +y)\right)=0,\quad (E+cP)''\left(Q_{c}(\cdot +y)\right)=L_{c,y},\\
&h^{cs}(y, 0, 0)=0, \quad Dh^{cs}(y,0, 0).
\end{align*}
Therefore, due to the smoothness of $h^{cs}$, from Lemma \ref{decompvb} one has the following expansion for $[0, T]$, 
\[
\begin{split}
& \wt E\big(y, a^+, a^-, V^{e}\big)-\wt E(y,0, 0, 0)\\
= & \frac{1}{2}\langle L_{c,y} (V^{e} + a^- V_c^-), V^{e} + a^- V_c^-\rangle+ O(\Vert V^{e}(t)\Vert^{3}_{H^{1}} + |a^-|^3)\\
= & \frac{1}{2}\langle L_{c,y} V^{e}, V^{e} \rangle+ O\big(\Vert V^{e}(t)\Vert^{3}_{H^{1}} + |a^-|^3\big)\\
\ge & (1/C)\Vert V^{e} \Vert^{2}_{H^{1}}- C(\Vert V^{e}\Vert^{3}_{H^{1}} + |a^-|^3). 
\end{split}\]
On the other hand, by the conservation and the translation invariance of the energy-momentum functional, we have 
\[
\begin{split}
& \wt E\big((y, a^+, a^-, V^{e})(t)\big)-\wt E\big(y(t), 0, 0, 0\big) = \wt E(\bar y, \bar a^+, \bar a^-, \bar V^e\big) -\wt E(\bar y, 0, 0, 0)\\
\le & C( \Vert \bar V^{e}\Vert^{2}_{H^{1}} + |\bar a^-|^3).
\end{split}\]
These inequalities imply 
\[
\Vert V^e\Vert_{H^1}^2 \le C (\Vert \bar V^e \Vert_{H^1}^2 + |\bar a^-|^3 + |a^-|^3).
\]
It is straight forward to obtain from \eqref{E:temp-4} and \eqref{E:hpm1} 
\[
\Vert V^e\Vert_{H^1}^2 \le C (\Vert \bar V^e \Vert_{H^1}^2 + |\bar a^- - h^-(\bar y, \bar V^e)|^3), \quad t\in [0, T]. 
\]
By choosing $C_0$ appropriately, the above inequality implies $\Vert V^e\Vert_{H^1}<\delta$, then $T$ may be extended to $+\infty$ and we obtain the desired estimates on $\CW^{cs}(\CM_c)$. 
\end{proof}

\begin{remark} \label{R:cs}
The same proof as above actually implies that the estimates in Propositions \ref{offcs} \ref{offcu}, and \ref{attract-center} hold for any $\wt h^{cs}, \wt h^{cu} \in \Gamma_{\mu, \delta} \cap C^2$ if $Em^\perp \big(graph (h^{cs, cu})\big)$ are locally invariant under \eqref{eqintf}, {\it without} modification by cut-off.  
\end{remark}

This proposition implies the orbital stability of $\CM_c$ inside $\CW^{cs} (\CM_c)$ and the exponential stability of $\CW^c (\CM_c)$ inside $\CW^{cs} (\CM_c)$ as $t\to +\infty$. Parallel results hold for the center-unstable manifold $\CW^{cu}(\CM_c)$ as $t\to \infty$. Moreover, $\CM_c$ is orbitally stable inside $\CW^c(\CM_c)$ as $ t \to \pm \infty$. The estimates in Propositions \ref{offcs}, \ref{offcu}, and \ref{attract-center}  yield the following characterizations. 

\begin{proposition} \label{characterization}
There exists $\delta>0$ and $C>1$, such that, for any $U_0\in H^1$ satisfying $dist(U_0, M_c) \le \delta$
\begin{enumerate}
\item $U_0\in \CW^{cs} (\CM_c)$ if and only if the solution $U(t)$ to \eqref{eqintf} with $U(0)=U_0$ satisfies $dist\big(U(t), \CM_c\big) \le C\delta$ for all $t\ge 0$. 
\item $U_0\in \CW^{cu} (\CM_c)$ if and only if the solution $U(t)$ to \eqref{eqintf} with $U(0)=U_0$ satisfies $dist\big(U(t), \CM_c\big) \le C\delta$ for all $t\le 0$.
\item $U_0\in \CW^{c} (\CM_c)$ if and only if the solution $U(t)$ to \eqref{eqintf} with $U(0)=U_0$ satisfies $dist\big(U(t), \CM_c\big) \le C\delta$ for all $t\in \R$. 
 \end{enumerate} 
\end{proposition}

\begin{remark} \label{R:unique}
Usually, $\CW^{cs}(\CM_c)$,  $\CW^{cu}(\CM_c)$, and $\CW^{c}(\CM_c)$ may not be unique due to the cut-off modification in their constructions. However,  the above characterization implies the uniqueness of the local center-stable, center-unstable, and the center manifolds of $\CM_c$ under \eqref{eqintf}. 
\end{remark}

\subsection{Local Invariant Manifolds of $\mathcal{M}$}

So far we have constructed local invariant manifolds for $\CM_c$ with a fixed $c>0$. Recall the scaling transformation $\CT^\lambda$ defined in \eqref{E:scaling}. From Proposition \ref{characterization}, it is clear that the local invariant manifold $\CW^\alpha (\CM_c)$, $\alpha \in \{cs, cu, c\}$ for different $c$ differ only by a rescaling, namely, for any $c>0$, 
\begin{equation} \label{E:scaling-inma}
\CW^{cs, cu, c} (\CM_c) = \{ \CT^{\sqrt{c}} U \mid U \in \CW^{cs, cu, c} (\CM_1)\}.  
\end{equation}
 
The following lemma indicates that local invariant manifolds of $\CM_c$ for nearby $c$ patch perfectly. Let $Em_c^\perp$ denote the embedding defined in \eqref{E:em-perp} for $c$. 

\begin{lemma}
For any $c>0$, there exists $\epi= \epi(c)>0$ such that if $|c_j- c|\le \epi$, $j=1,2$, then for $\alpha \in \{cs, cu, c\}$, 
$$\CW^{\alpha}(\CM_{c_2})\cap Em_{c_2}^\perp \left(B^{1}(\epi)\oplus B^1(\epi)\oplus \CX_{c_2}^e (\epi)\right) \subset \CW^{\alpha}(\CM_{c_1}).$$
 \end{lemma}
 
\begin{proof}
$Q_{c_2}(x- c_2t)$ is a solution to \eqref{gkdv}. In the traveling frame $(t,x-c_1 t)$, it becomes $Q_{c_2}\left(x-( c_2-c_1)t\right).$ 
Since 
\[
\Vert Q_{c_2}\left(x-(c_2-c_1)t\right)-Q_{c_1}\left(x-( c_2-c_1)t\right)\Vert_{H^1}\le C|c_2- c_1|
\] 
and $Q_{c_1}\left(x-(c_2-c_1)t\right)\in \CM_{c_1}$,
we have 
$$\inf_{y\in \R}\Vert Q_{c_2}\left(\cdot-(c_2-c_1)t\right)-Q_{c_1}\Vert_{H^1}\le C|c_2-c_1|.$$ 
Then the desired result follows by Proposition \ref{characterization}. 

\end{proof}

Consequently, we can patch $\CW^{cs, cu, c}(\mathcal{M}_c)$ with different $c$ to form the center-stable manifold of $\CM$. In fact, let 
\begin{equation}
\CW^{cs, cu, c}(\CM)= \bigcup_{c>0}\CW^{cs, cu, c}(\CM). 
\end{equation}

The above lemma implies that $\CW^{cs, cu, c} (\CM)$ is a smooth codim-1 submanifold in $H^1$.

\begin{remark} \label{R:scalingCS}
In fact, due to the scaling invariance \eqref{E:scaling-inma}, 
$\CW^{cs, cu, c}(\CM_c)$ are invariant under the rescaling \eqref{E:scaling}.
\end{remark}

\begin{remark}  \label{R:stable-unstable}
The stable and unstable manifolds can be constructed through a {\it simpler} procedure. Thanks to their uniqueness, one may construct stable and unstable manifolds of a single solitary wave $Q_c$, and then those of $Q_c(\cdot+y)$ can be obtained simply by translation. In this procedure, 
since $y=0 \in \R^3$ is fixed in the construction, the only obstacle preventing the classical invariant manifold theory to be applicable straightforwardly is the derivative loss in the nonlinearity, which can be overcome by the smoothing estimates in Section \ref{LinearAnalysis}. Actually, with the smoothing estimates, one may carry out the construction following the approach in \cite{CL}.  
\end{remark}

\begin{remark}
In \cite{Com}, Combet constructed solutions converging to solitary waves. From the point of view of dynamical systems theory, the solutions constructed by Combet must locate in the stable manifolds of solitary waves. 
\end{remark}

\bibliography{kdv}{}
\bibliographystyle{plain}

\end{document}